\documentclass{amsart}

\usepackage{amsmath}
\usepackage{amssymb}
\usepackage{amsthm}
\usepackage{subcaption}
\usepackage{comment}
\usepackage{hyperref,color}
\usepackage{graphicx}

\graphicspath{{Figures/}} 
\graphicspath{{./Figures/}}

\def\cA{\mathcal{A}}
\def\cB{\mathcal{B}}
\def\cC{\mathcal{C}}
\def\cD{\mathcal{D}}
\def\cE{\mathcal{E}}
\def\cZ{\mathcal{Z}}
\def\tcC{\widetilde{\cC}}
\def\tcA{\widetilde{\cA}}

\def\Ai{\mathcal{A}_{\mathrm{map}}}

\def\kk{\mathbf{k}}

\def\cAk{\cA}
\def\tcAk{\widetilde{\cA}}
\def\Ak{A}
\def\tAk{\widetilde{A}}
\def\Pk{P_k}

\def\tA{\widetilde{A}}

\def\Sym{\mathrm{Sym}}
\def\Seq{\mathrm{Seq}}
\def\Set{\mathrm{Set}}
\def\MSet{\mathrm{MSet}}
\def\Cyc{\mathrm{Cyc}}

\def\Int{\mathrm{Int}}

\def\od{\overline{d}}

\def\Gammal{\frownacc{\Gamma}}

\def\dtv{d_{\mathrm{TV}}}

\newcommand{\adjustedaccent}[1]{%
  \mathchoice{}{}
    {\mbox{\raisebox{-.35ex}[0pt][0pt]{$\scriptstyle#1$}}}
    {\mbox{\raisebox{-.35ex}[0pt][0pt]{$\scriptscriptstyle#1$}}}
}

\newcommand\frownacc[1]{\overset{\adjustedaccent{\smallfrown}}{#1}}

\newtheorem{prop}{Proposition}
\newtheorem{lem}[prop]{Lemma}

\theoremstyle{remark}
\newtheorem{remark}{Remark}
\newtheorem{definition}{Definition}

\begin{document}

\title{Leap generators for composition schemes}


\author[\'Eric Fusy and Carine Pivoteau]{\'Eric Fusy$^{*}$ and Carine Pivoteau$^{*}$}
\thanks{$^{*}$LIGM, CNRS, Univ Gustave Eiffel, ESIEE Paris, F-77454 Marne-la-Vall\'ee, France, eric.fusy@univ-eiffel.fr, carine.pivoteau@univ-eiffel.fr.
	}

\begin{abstract}
Leap generators have been introduced in [Duchon et al.'04] for exact-size random generation of structures in a class of the form $\cC=\Seq(\cB)$ (sequence construction), in the supercritical case. We extend these generators to supercritical composition schemes $\cC=\cA\circ\cB$. Compared to the sequence construction, the obtained exact-size random generator for $\cC$ still has linear time complexity (under conditions on the sampling complexity in $\cA$ and $\cB$), but perfect uniformity of the distribution is lost in general. However the distribution on $\cC_n$, called leap distribution, is asymptotically uniform, the total variation distance from the uniform distribution being $(c+o(1))n^{-1/2}$ for an explicit constant $c$. These generators are simple to implement and can be applied to several classes of walks and trees, in particular P\'olya trees. 
Leap generators can also be given for certain critical composition schemes, those relating planar map families, where this time the total variation distance to the uniform distribution is $\sim c\,n^{-1/3}$ for an explicit constant $c$. 
\end{abstract}

\smallskip

\maketitle


\section{Introduction}
A random generation algorithm aims to sample uniformly at random, and as efficiently as possible, an object of a prescribed size $n$ from a combinatorial class (e.g. of trees, permutations, partitions, graphs, maps, ...).
An important motivation is to obtain simulations of large random objects and formulate conjectures about their asymptotic behavior, in particular for parameters that are difficult to analyze. Fast running time (ideally linear) is thus paramount.  
Markov chain Monte Carlo (MCMC) algorithms~\cite{levin2017markov,craiu2026handbook} give a robust and flexible approach to random generation; however their mixing time can be difficult to analyze (and the complexity is often superlinear). 
Regarding combinatorial methods, bijective random samplers are typically fast, but require a  correspondence to a class for which counting and random generation are easier; such generators are known for instance for classes of trees (e.g. via Pr\"ufer code), planar maps~\cite{Schaeffer99}, and  automata~\cite{bassino2007enumeration}.   
On the other hand, two automatic methods of random generation have been developed for classes that admit a decomposition in terms of the basic constructions of symbolic combinatorics~\cite{flajolet2009analytic}. (i) The recursive method~\cite{flajolet1994calculus,denise1999uniform} based on counting coefficients achieves exact-size sampling, but at the cost of superlinear precomputations that limit the size to the range of thousands. (ii) Boltzmann sampling~\cite{duchon2004boltzmann,flajolet2007boltzmann}, based on evaluations of the generating functions at a well-chosen value, often achieves linear-time complexity, however at the cost of a fluctuating output size. For approximate-size sampling, Boltzmann samplers perform very well and make it possible to draw random objects of size in the range of millions. However, in practice it is convenient in simulations to have very large samples at certain sizes (e.g. powers of two) in order to obtain precise histograms of parameters, such as distance parameters, and to measure how the distribution scales with the size, see~\cite[Sec.~3]{barkley2019precision}. In that context Boltzmann samples are not really well suited, as they require heavy rejection to be tuned at an exact size, which often turns the complexity from linear to quadratic. 

To remedy this problem, several techniques have been proposed in order to extract exact-size random samplers of linear-time complexity from Boltzmann samplers, with the challenge of making the overall cost of sampling (including failed attempts) linear. For Galton-Watson trees, Devroye's algorithm~\cite{devroye2012simulating} achieves this by reducing the cost of each failed attempt to $\widetilde{O}(1)$, while requiring $\Theta(\sqrt{n})$ attempts (the final successful attempt having an additional cost $\Theta(n)$). On the other hand,  
several random generation algorithms have been developed with the idea (originating from~\cite{vonNeumann}) of boosting the acceptance probability of the drawn structures, so as to make the probability of success of every attempt $\Omega(1)$; such algorithms have been successfully applied to integer partitions and set partitions~\cite{arratia2016probabilistic} (see also~\cite{carayol} which combines this with Devroye's approach), to structures with a context-free specification~\cite{sportiello}, and very recently to P\'olya trees~\cite{stuflerPolya}. Finally, for classes of the form $\cC=\Seq(\cB)$ (sequence construction) in the so-called supercritical case, a simple ``leap" process~\cite[Sec.7.1]{duchon2004boltzmann} built on a Boltzmann sampler for~$\cB$ makes it possible to have acceptance probability $\Omega(1)$ for each attempt.

In this article, we extend the latter technique of leap generators to the more general setting of a supercritical composition scheme $\cC=\cA\circ\cB$, and show that they still perform in  linear time for exact-size sampling, assuming $\cA$ (resp. $\cB$) has an exact-size (resp. a Boltzmann) sampler whose complexity is linear in the size of the output. In contrast to the sequence construction, perfect uniformity is lost in general, but the obtained leap generators are asymptotically uniform, in the sense that the total variation distance to the uniform distribution is $o(1)$: we prove more precisely that it is $(c+o(1)) n^{-1/2}$ for some explicit $c$.  This is sufficient to observe the limit behavior  of parameters, as limit laws are the same under an asymptotically uniform distribution as under the uniform distribution. Furthermore, via an additional rejection step, which we call acceleration of convergence, we show that the total variation distance to the uniform distribution can be improved to $O(n^{-r})$ for any fixed $r$. 
When applicable, our leap generators are very easy to implement and, like Boltzmann samplers, they make simple random choices based on precomputed values and require very little auxiliary memory.
They can be applied to a variety of structures, notably classes of unlabeled trees such as P\'olya trees and unlabeled phylogenetic trees (i.e, non-plane binary trees), and they extend to certain critical composition schemes such as those of planar maps~\cite{banderier2001random}, where this time the total variation distance to the uniform distribution is $(c+o(1)) n^{-1/3}$ for some explicit $c$. We provide several simulations, inspired by those recently performed in~\cite{bartholdi2024algorithm}, showing the practical efficiency of our approach, and illustrating that the empirical (or exact) laws of parameters ---such as the height of Motzkin walks or of P\'olya trees--- are very close under the leap distribution and under the uniform distribution, and are nearly indistinguishable when using first-order acceleration of convergence.       

We note that the idea of relaxing perfect uniformity to asymptotic uniformity ---in random sampling and analysis of random discrete structures--- already appeared in different contexts. In MCMC generators the uniform distribution at a given size is only reached asymptotically in the number of steps. In practice one performs a number of steps large enough to be close to the uniform distribution. This required number of steps (mixing time) is often hard to analyze, but it can also be validated experimentally, as for instance in the recent random generator~\cite{bartholdi2024algorithm} of P\'olya trees. Regarding the bijective approach, in the context of fixed-genus factorizations of an $n$-cycle into transpositions, an asymptotically uniform bijective random generator  has been developed in~\cite{feray2022random}, and used to prove convergence in law of a chord process associated  to the factorization.   
Asymptotically uniform random samplers have also been developed for graphs with prescribed vertex-degrees~\cite{steger1999generating,kim2003generating,bayati2010sequential}, based on switching operations on multigraphs to get rid of loops and multiple edges. Note that even in the context of theoretically uniform samplers (e.g. Boltzmann samplers), typical practical implementations are not perfectly uniform, as they rely on fixed precision evaluation. For a sufficiently high precision  chosen depending on the output size, this gives an asymptotically negligible bias from the uniform distribution. Finally, in the context of analyzing random discrete structures, it is known that establishing a limit law (e.g. scaling limit) for random structures in an unrooted class can often be reduced to the easier task of establishing the same limit law for an associated rooted class, if one can prove that rooting is asymptotically unbiased in the total variation sense, i.e., that the distribution induced by rooted objects on unrooted objects is asymptotically uniform. This property is known for various families of planar maps, since planar maps with a non-trivial automorphism are exponentially rare~\cite{liskovets81,richmond1995almost}. It has also been recently established for free trees~\cite{stufler2026probabilistic}, where it is shown that the total variation distance between the uniform distribution on free trees and the distribution induced by P\'olya trees is $n^{-a}$ for any $a<1/2$.

\section{Background and tools}\label{sec:background_tools}
\subsection{Symbolic combinatorics}
We recall that a combinatorial class $\cC$ is a set of objects equipped with a size function $\gamma\to |\gamma|\in\mathbb{N}$, such that for each $n\geq 0$ the set $\cC_n=\{\gamma\in\cC\ |\ |\gamma|=n\}$ is finite, each object $\gamma\in\cC_n$ being made of $n$ atoms (e.g. nodes of a tree, steps of a walk...). In the unlabeled setting the atoms carry no label, and the associated (ordinary) generating function is $C(z)=\sum_n |\cC_n|z^n$; while in the labeled setting the atoms carry distinct labels in $[1..n]$ and the associated (exponential) generating function is $C(z)=\sum_n \frac1{n!}|\cC_n|z^n$. 

\subsubsection{Dictionary of symbolic combinatorics}
We use the usual dictionary of symbolic combinatorics~\cite{flajolet2009analytic}, with the two base classes $\mathbf{1}$ and $\mathcal{Z}$ of respective generating functions $z\to 1$ and $z\to z$, and the operators $+$ (disjoint union), $\times$ (product). We will also use the constructions for sequence, cycle, and set of structures, and their restriction with a prescribed number of components, all these constructions being defined both in the labeled and unlabeled setting 
(in the unlabeled setting these constructions allow the same component to appear multiple times).  

Finally, of crucial importance for our method is the substitution operator: for two labeled combinatorial classes $\cA,\cB$ where $\cB$ has no object of size $0$, the \emph{composed class}  
\[
\cC=\cA\circ\cB
\] 
is the set of objects $\gamma$ that are obtained from an object $\alpha\in\cA$ and a sequence  $\beta_1,\ldots,\beta_{|\alpha|}$, where atom $i$ of $\alpha$ is considered as ``substituted"  by $\beta_i$; the size of $\gamma$ is $n:=|\beta_1|+\cdots+|\beta_k|$, and a further relabeling function assigns distinct labels in $[1..n]$ to the atoms (the atom-set of $\gamma$ being the union of the atom-sets from $\beta_1,\ldots,\beta_k$). 
The integer $k$ is called the \emph{core-size} of $\gamma$. In the unlabeled case the substitution operator is well defined if the atoms of objects in $\cA$ are unambiguously distinguishable (e.g. the steps of a walk, or the nodes of a binary tree), so that one can rank the atoms, even without them carrying a label. In both the labeled and unlabeled case (when applicable) the generating functions are related by 
\[
C(z)=A(B(z)).
\] 

\subsubsection{Unlabeled class and symmetry class associated to a labeled class}
We consider labeled combinatorial classes $\cC$ in the usual setting where relabeling\footnote{We refer to~\cite{bergeron1998combinatorial} for a more formal setting to define combinatorial classes and their relabeling action; we adopt here a lighter formalism as in~\cite{flajolet2009analytic}.} the atoms of an object in $\cC_n$ by any permutation in $\mathfrak{S}_n$ yields an object in $\cC_n$, so that $\mathfrak{S}_n$ acts on $\cC_n$. We let $\tcC_n=\cC_n/\mathfrak{S}_n$ be the set of orbits of 
$\cC_n$ under the relabeling action; then $\tcC=\cup_n\tcC_n$ is the \emph{unlabeled class} for $\cC$. 
On the other hand, we let $\Sym(\cC_n)$ be the set of pairs $(\sigma,\gamma)$ in $\mathfrak{S}_n\times \cC_n$ such that $\sigma\cdot\gamma=\gamma$ (in which case $\sigma$ is called an automorphism of $\gamma$). Then the labeled class $\Sym(\cC)=\cup_n\, \Sym(\cC_n)$ is called the \emph{symmetry class} for $\cC$. For $(\sigma,\gamma)\in\Sym(\cC_n)$, the underlying unlabeled object is defined as the orbit $\widetilde{\gamma}\in\tcC_n$ to which $\gamma$ belongs. Burnside's lemma ensures that 
\[
\Sym(\cC_n)\ \simeq\  n!\ \tcC_n, 
\]
i.e., every $\widetilde{\gamma}\in\tcC_n$ is the underlying object for $n!$ pairs in $\Sym(\cC_n)$.
Thus uniform random sampling in $\tcC_n$  reduces to uniform random sampling in $\Sym(\cC_n)$ (in our case we will use this reduction for asymptotically uniform random sampling).

\subsection{Singular behaviors and composition schemes}\label{sec:singular}
Let $\cC$ be a combinatorial class, with $C(z)$ the associated generating function, and $\rho$ the radius of convergence of $C(z)$, assumed to be strictly positive and finite. The \emph{support} of $C(z)$ is the integer set $\{n\geq 0\ |\ [z^n]C(z)\neq 0\}$. 
The generating function $C(z)$ is called \emph{periodic} if there exist $d\geq 2, c\geq 0$ such that the support of $C(z)$ is included in $c+d\,\mathbb{N}$. In that case, with $\omega=e^{2i\pi/d}$, we have $C(z)$ singular at $\rho\,\omega^i$ for all $i\in[0..d-1]$,  so that $\rho$ is not the unique dominant singularity.   
For $s\in \mathbb{R}$ with $-s\notin \mathbb{N}$, we say that $C(z)$ has \emph{singular exponent}  $s$ if $\rho$ is its unique dominant singularity, and there is a non-zero constant $\kappa$ and a polynomial $P(z)$ such that
\begin{equation}\label{eq:sing_exp}
C(z)=P(z)+\kappa(1-z/\rho)^{-s}+O((1-z/\rho)^{-s+1}),
\end{equation}
where the expansion holds in a $\Delta$-neighborhood of $\rho$. In that case, transfer theorems of analytic combinatorics~\cite[Section VI.3]{flajolet2009analytic} ensure that 
\[[z^n]C(z)= \frac{\kappa}{\Gamma(s)}\rho^{-n}n^{s-1}(1+O(1/n)).\]  
For example, the singular exponent for the generating functions of regular languages is typically a positive integer, for tree families it is typically $-1/2$, and for planar map families it is typically $-3/2$. 

For a composition scheme $\cC=\cA\circ\cB$, we let $\rho_A,\rho_B,\rho_C$ be the radii of convergence of the respective generating functions $A(y), B(z), C(z)$. Following~\cite{Go95}, the scheme is called \emph{subcritical} if $\rho_C=\rho_B$ and $A(y)$ is analytic at $B(\rho_C)$ (i.e., the singularity of $C(z)$ is due to $B(z)$), and it is called \emph{supercritical} if $\rho_B>\rho_C$, in which case one must have
 $B(\rho_C)=\rho_A$ (i.e., the singularity of $C(z)$ is due to $A(y)$). 
 It is called \emph{aperiodic} if $A(y),B(z),C(z)$ are aperiodic. 
 A supercritical aperiodic scheme is said to be of singular exponent $s$ if $A(y)$ has singular exponent $s$. 
 In that case, a standard calculation ensures that $C(z)$ is also\footnote{Uniqueness of the dominant singularity of $C(z)$ is granted by the fact that $B(z)$ is aperiodic and by the Daffodil lemma~\cite[Lemma.IV.1]{flajolet2009analytic}.} of singular exponent $s$ and moreover the respective values $\kappa=\kappa_A,\kappa=\kappa_C$ in the singular expansion~\eqref{eq:sing_exp} are related by
 \begin{equation}\label{eq:kappa}
\frac{\kappa_A}{\kappa_C}=\mu^s\ \ \mathrm{where}\ \mu=\frac{\rho B'(\rho)}{B(\rho)}.
\end{equation} 

\begin{remark}
Random structures in a class $\cC=\cA\circ\cB$ are studied in a probabilistic setting under the name of Gibbs partitions~\cite{pitman}, i.e., random partitions $\pi=B_1\cup\cdots\cup B_k$ of $[1..n]$ with weight proportional to $v_k\prod_{j=1}^k w_{|B_j|}$, for some nonegative fixed weight-vectors $(v_1,v_2\ldots)$ and $(w_1,w_2,\ldots)$. 
The connection with labeled composition schemes is done by $v_k=|\cA_k|/k!$ and $w_j=|\cB_j|$.
The behavior of random Gibbs partitions is known~\cite{stufler2018gibbs,stufler2024gibbs,Go95} to differ markedly between the subcritical case and the supercritical case.  
In the subcritical case, the partition typically has $\Theta(1)$ parts, the largest one of size $n-O(1)$; while in the supercritical case the partition typically has $\Theta(n)$ parts whose sizes (asymptotically) behave as independent copies of a random variable having an  exponentially decaying tail.    
\end{remark}

\begin{definition}
 A composition scheme is called \emph{admissible} if it is supercritical, 
  aperiodic, and the support of $A(y)$ contains $\mathbb{N}^*$. 
\end{definition}

\subsection{Boltzmann samplers}
For an unlabeled (resp. a labeled) combinatorial class $\cC$, with generating function $C(z)$,  and for $x>0$ such that $C(x)$ converges,  the \emph{Boltzmann distribution} $\mathbb{P}_x^{\,\cC}$ is the probability distribution on $\cC$ given by
\begin{equation}
\mathbb{P}_x^{\,\cC}(\gamma)=\frac1{C(x)}x^{|\gamma|}\ \mathrm{(unlabeled\ case)},\ \ \ \ \ \mathbb{P}_x^{\,\cC}(\gamma)=\frac1{C(x)}\frac{x^{|\gamma|}}{|\gamma|!}\ \mathrm{(labeled\ case).}
\end{equation}
For $\gamma$ drawn under $\mathbb{P}_x^{\,\cC}$, the expectation $\mu(x)$ and standard deviation $\sigma(x)$ of the random variable $|\gamma|$ are given by
\begin{equation}\label{eq:mu_sigma}
\mu(x)=\frac{xC'(x)}{C(x)},\ \ \ \ \sigma(x)=\sqrt{x\mu'(x)}.
\end{equation}

A Boltzmann sampler~\cite{duchon2004boltzmann} is a procedure $\Gamma C(x)$ that draws a random object in $\cC$ under the distribution $\mathbb{P}_x^{\,\cC}$, for any fixed $x>0$ such that $C(x)$ converges. It is said to be of \emph{linear time complexity} if the runtime of every generation is linearly bounded by the size of the output object. 

Explicit sampling rules allow to obtain linear time Boltzmann samplers for combinatorial classes that are recursively specified from the base classes in terms of  the standard constructions. These rules are described in~\cite{duchon2004boltzmann} and complemented in~\cite{Fusy09} (substitution operator) and in~\cite{flajolet2007boltzmann,bodirsky2011boltzmann} (set and cycle constructions in the unlabeled setting).

\subsection{Asymptotically uniform random sampling}
For a finite set $E$ and two probability distributions $\pi,\pi'$ on $E$, the total variation distance between $\pi$ and $\pi'$ is defined as 
\[
\dtv(\pi,\pi'):=\frac1{2}\sum_{x\in E}|\pi(x)-\pi'(x)|.
\]
For $\cC=\cup_n\cC_n$ a combinatorial class, a random generator $\Gamma C[n]$ on $\cC_n$ is called \emph{asymptotically uniform} if the total variation distance between the uniform distribution $\pi_n$ on $\cC_n$ and the distribution $\pi_n'$ induced by $\Gamma C[n]$ satisfies $\dtv(\pi_n,\pi_n')=o(1)$. 

\begin{remark}\label{rk:asympt_uniform_parameter}
If we consider a parameter $X=X_n$ on $\cC_n$ (i.e., a function from $\cC_n$ to $\mathbb{R}$), and let $\pi_{X,n}$ (resp. $\pi'_{X,n}$) be the law of
 $X_n$ under $\pi_n$ (resp. under $\pi_n'$), then we obviously have $\dtv(\pi_{X,n},\pi'_{X,n})\leq \dtv(\pi_{n},\pi_{n}')$. Hence, if $\dtv(\pi_n,\pi_n')=o(1)$ then $\dtv(\pi_{X,n},\pi_{X,n}')=o(1)$, possibly with smaller order of magnitude.
\end{remark}

\section{Leap generators}
\subsection{Description of the generator}
We consider an admissible composition scheme
\[
\cC=\cA\circ \cB,
\]
and assume we have an exact-size uniform random generator $\Gamma A[k]$ for the class $\cA$, and a Boltzmann sampler $\Gamma B(x)$ for the class $\cB$.  
Letting $\rho$ be the radius of convergence of $C(z)$, the \emph{leap generator} $\Gammal C[n]$ proceeds as follows (see also Figure~\ref{fig:leap}).

\begin{figure}
\begin{center}
\includegraphics[width=10cm]{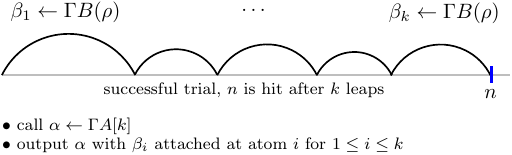}
\end{center}
\caption{The leap generator process.}
\label{fig:leap}
\end{figure} 

At each trial we generate a sequence $(\beta_i)_{i\geq 1}$ of objects in $\cB$ one by one,  each object being drawn (independently) by a call to $\Gamma B(\rho)$, and we stop at the largest~$k$ such that the total size $S:=|\beta_1|+\cdots+|\beta_k|$ is at most $n$. If $S<n$ the trial fails and we restart. 
If $S=n$, we call  $\alpha\leftarrow \Gamma A[k]$, and for $i\in[1..k]$ we substitute the $i$th atom of $\alpha$ by $\beta_i$. 
We return the composed object $\gamma\in\cC_n$ thus obtained (if we are in the labeled setting, randomly permuted labels in $[1..n]$ are assigned to the atoms of $\gamma$).  

\begin{remark}
In the special case $\cA=\Seq$, the leap generator for $\cC=\Seq(\cB)$ is exactly the leap generation process for $\cC$ as described in~\cite[Section 7.1]{duchon2004boltzmann}. In that case it gives a uniform random generator on $\cC_n$.
\end{remark}

\subsection{Analysis of the generator}
We denote respectively by $\mu$ and $\sigma$ the expectation and standard deviation of the size of $|\Gamma B(\rho)|$, as given by~\eqref{eq:mu_sigma}. 
\begin{prop}\label{prop:linear}
Let $q_n$ be the probability of success at each trial in $\Gammal C[n]$. Then
\begin{equation}\label{eq:estimate_qn}
q_n=[z^n]\frac{1}{1-B(\rho z)/B(\rho)}=\frac1{\mu}+O(r^n)\ \mathrm{for\ some\ }r\in(0,1).
\end{equation}

If the exact-size random generator $\Gamma A[k]$ has time complexity $O(k)$ and the Boltzmann sampler has linear time-complexity, then the expected time complexity of $\Gammal C[n]$ is $O(n)$. 
\end{prop}
\begin{proof}
We give the proof for $q_n$ in the unlabeled case (the labeled case proof is very similar). 
A trial is a success if it generates a sequence $(\beta_1,\ldots,\beta_k)$ of total size $n$. For each given such sequence, the probability that it is generated is $\prod_{i=1}^k\mathbb{P}_{\rho}^{\cB}(\beta_i)$, with $\mathbb{P}_{\rho}^{\cB}(\beta)=\rho^{|\beta|}/B(\rho)$.  
Hence the probability that the sequence is generated is $\rho^n/B(\rho)^k$. The total contribution of all structures of size $n$ in $\cB^k$ is thus  $\rho^n[z^n]B(z)^k/B(\rho)^k=[z^n]B(\rho z)^k/B(\rho)^k$. Summing over $k\geq 0$ we find that the total contribution (which is $q_n$) is $[z^n]Q(z)$, 
with  $Q(z):=\frac{1}{1-B(\rho z)/B(\rho)}$, as claimed. Since the generating function $B(z)$ is aperiodic, by the Daffodil lemma the maximum value of  $|B(z)|$ on $|z|\leq \rho$ is uniquely attained at $z=\rho$. Thus the meromorphic function $Q(z)$ has a unique pole on $|z|\leq 1$, at $z=1$.  Moreover, we clearly have
\[
\lim_{z\to 1}(1-z)Q(z)=\frac{B(\rho)}{\rho B'(\rho)}=\frac1{\mu}.
\] 
Transfer theorems~\cite[Section VI.3]{flajolet2009analytic} thus ensure that $q_n=\frac1{\mu}+O(r^n)$ for some $r\in(0,1)$ (taking $r$ such that $Q(z)$ has no pole $\neq 1$ on the disk $|z|\leq 1/r$).

Regarding the complexity statement, by a standard formula for rejection sampling, the overall expected cost of $\Gammal C[n]$ is equal to the expected cost of one trial divided by $q_n$. Since $q_n\to 1/\mu$, it remains to show that the expected cost of one trial is $O(n)$. We use the above notation for $k,S,\beta_i$ in the description of $\Gammal C[n]$. Let $\tilde{S}:=S+|\beta_{k+1}|$.  
Since the samplers $\Gamma A[k]$ and $\Gamma B(x)$ have linear time complexity,  the expected cost of one trial is $O(\mathbb{E}(\tilde{S}))+O(k)$, the second term given by the possible call to $\Gamma A[k]$. By definition of the stopping time, $S\leq n$ and 
$\tilde{S}>n$. Note that $k\leq n$ (as $\cB$ has no object of size $0$). 
Hence, letting $M_n:=\mathrm{max}(|\beta_1|,\ldots,|\beta_{n}|)$, we have
$\tilde{S}\leq n+M_{n+1}$.  Since $|\Gamma B(\rho)|$ has an  exponentially decaying tail, we classically have $\mathbb{E}(M_n)=O(\log(n))$, 
so that $\mathbb{E}(\tilde{S})\sim n$. Hence the expected cost of one trial is $O(n)$.  \end{proof}

The probability distribution on $\cC_n$ given by $\Gammal C[n]$ is called the \emph{leap distribution} on $\cC_n$ and is denoted by $\pi_n'$, while the uniform distribution on $\cC_n$ is denoted by $\pi_n$.  We also use the notation $\delta_n$ to be $1$ in the unlabeled setting and $1/n!$ in the labeled setting.

\begin{remark}
It is also possible to perform the leap process a single time, without rejection: one then calls $\Gamma A[k]$ with $k$ number of leaps done before stopping. With the notation used to describe $\Gammal C[n]$, the obtained structure has a random size $Z_n=S\leq n$ very close to $n$. Indeed, letting $Z$ be the size of $\Gamma B(\rho)$, for any fixed $i\geq 0$ one has 
\[
\mathbb{P}(n-Z_n=i)=q_{n-i}\ \!\mathbb{P}(Z>i)\sim \frac1{\mu}\mathbb{P}(Z>i).
\] 
Thus $n-Z_n$ converges to the discrete law $p(i)=\frac1{\mu}\mathbb{P}(Z>i)$, which has an exponential tail.  
Moreover, the distribution induced on any fixed size $n' \leq n$ is the leap distribution at that size.  
\end{remark}


\begin{lem}\label{lem:distort}
Let $a_k=[y^k]A(y)=\delta_k|\cA_k|$ and $c_n=[z^n]C(z)=\delta_n|\cC_n|$. Let $\cC_{n,k}$ be the set of structures in $\cC_n$ having core-size $k$.  
Then for every $\gamma\in\cC_{n,k}$ one has
\begin{equation}\label{eq:pinprime}
\pi_n'(\gamma)=\frac{1}{|\cC_n|}d_{n,k},\ \ \mathrm{with}\ \ d_{n,k}:=\frac1{q_n}\frac{c_n\rho^n}{a_kB(\rho)^k}.
\end{equation} 
\end{lem}
\begin{proof}
Letting $\pi_n^{\mathrm{first}}(\gamma)$ be the probability that $\gamma$ is generated at the first trial, we clearly have 
\[
\pi_n'(\gamma)=\sum_{i\geq 1}\pi_n^{\mathrm{first}}(\gamma)\,(1-q_n)^{i-1}=\frac{\pi_n^{\mathrm{first}}(\gamma)}{q_n},
\]
where the $i$th term is the probability that $\gamma$ is generated at the $i$th trial (after failure of the first $i-1$ trials). 

In the unlabeled case, let $\alpha\in\cA_k$ and $(\beta_1,\ldots,\beta_k)\in(\cB^k)_n$ be the structures composing $\gamma\in\cC_n$. Then the probability that the leap process outputs $\beta_1,\ldots,\beta_k$ is $\rho^n/B(\rho)^k$ (as seen in the proof of Proposition~\ref{prop:linear}), and then the probability that $\alpha$ is drawn from $\Gamma A[k]$ is $1/a_k$. Hence
\[
\pi_n^{\mathrm{first}}(\gamma)=\frac1{a_k}\frac{\rho^n}{B(\rho)^k}.
\]
In the labeled case, for any given ordering (among $k!$ possibilities) of the $\cB$-components of $\gamma$, and for any ordering of the atoms among each of the $\cB$-components, we let $\alpha$ be the induced structure in $\cA_k$ (its atoms are labeled according to the ordering of the components). And let $\beta_1,\ldots,\beta_k$ be the $\cB$-components, where the atoms of each $\beta_i$ are labeled in $[1..|\beta_i|]$ according to their prescribed ordering. The probability that the leap process outputs $\beta_1,\ldots,\beta_k$ is equal to $\prod_{i=1}^k\frac1{B(\rho)}\frac{\rho^{|\beta_i|}}{|\beta_i|!}$. Then the probability that $\alpha$ is chosen by $\Gamma A[k]$ is equal to $1/|\cA_k|$. 
And finally the probability that the relabeling function assigns the correct labels to the atoms of $\gamma$ is $1/n!$. Since the number of possibilities for the orderings is $k!\prod_{i=1}^k|\beta_i|!$, we conclude that 
$\pi_n^{\mathrm{first}}(\gamma)=\frac1{n!}\frac{k!}{|\cA_k|}\prod_{i=1}^k\frac1{B(\rho)}\rho^{|\beta_i|}=\frac1{n!}\frac1{a_k}\frac{\rho^n}{B(\rho)^k}$.
\end{proof}
\begin{remark}\label{rk:pinp_coresize}
The quantity $d_{n,k}$, which is the ratio $\pi_n'(\gamma)/\pi_n(\gamma)$ for objects $\gamma\in\cC_{n,k}$, is called the \emph{distortion factor}. 
For a given large $n$, it tends to increase with $k$ when the singular exponent $s$ satisfies $s<1$, and it tends to decrease with $k$ when $s>1$. 
\end{remark}

\begin{lem}\label{lem:dtv}
Let $c_{n,k}=a_k[z^n]B(z)^k=\delta_n|\cC_{n,k}|$.  
Then
\begin{equation*}
\dtv(\pi_n,\pi_n')=\frac1{2}\sum_{k=1}^nS_{n,k},
\end{equation*}
where
\begin{align}
S_{n,k}&=\frac{c_{n,k}}{c_n}\left| 1 - d_{n,k} \right|\label{eq:Snk1}
\end{align}
\end{lem}
\begin{proof}
By~\eqref{eq:pinprime} we have $\dtv(\pi_n,\pi_n')=\frac1{2}\sum_{k=1}^nS_{n,k}$, where
\begin{align*}
S_{n,k}&=|\cC_{n,k}|\cdot\left| \frac1{|\cC_n|}-  d_{n,k}\frac1{|\cC_n|}\right|\\
&=\frac{|\cC_{n,k}|}{|\cC_n|}\cdot\left| 1 - d_{n,k} \right|\\
&=\frac{c_{n,k}}{c_n}\left| 1 - d_{n,k} \right|.
\end{align*}
\end{proof}

We will now prove that the leap generator is asymptotically uniform. The first step is to have estimates on the two factors $c_{n,k}/c_n$ and $|1-d_{n,k}|$ in $S_{n,k}$.  

\begin{lem}\label{lem:estimate_cnk}
Recall the notation $\mu,\sigma$ for the expectation and standard deviation of $|\Gamma B(\rho)|$.
Let $X_n$ be the random variable such that $\mathbb{P}(X_n=k)=\frac{c_{n,k}}{c_n}$. Then the  estimate
\begin{equation}
\label{eq:local_limit}
\mathbb{P}\Big(X_n=n/\mu+t\,\sigma\sqrt{n/\mu^3}\Big)\sim \frac{\sqrt{\mu^3}}{\sigma\sqrt{n}}\frac{1}{\sqrt{2\pi}}e^{-t^2/2}
\end{equation}
holds uniformly for $t$ in any fixed compact set.

Moreover, the following tail bound holds: there exist positive constants $\kappa,\xi,\zeta$ 
such that, for $n\geq 0$ and $0\leq t\leq \zeta\sqrt{n}$,
\begin{equation}
\label{eq:large_dev}
\mathbb{P}\Big(\big|X_n-n/\mu\big|\geq t\,\sigma\sqrt{n/\mu^3}\Big)\leq \xi\,e^{-\kappa\,t^2}.
\end{equation}
\end{lem}
\begin{proof}
The probability generating function $p_n(u)=\sum_k \mathbb{P}(X_n=k)u^k$ of $X_{n}$ is given by 
\[
p_n(u)=\frac{[z^n]C(z,u)}{[z^n]C(z)},\ \ \mathrm{with}\ C(z,u):=A(uB(z)).
\] 
Let $\rho(u)$ be the radius of convergence of $z\to C(z,u)$. Recall that $B(z)$ is analytic at $\rho$, with $B(\rho)=\rho_A$. Hence, $\rho(u)$ is given by the equation $B(\rho(u))=\rho_A/u$ for $u$ in a neighborhood of $1$. Moreover, by singularity analysis, we have the singular expansion
\begin{equation}\label{eq:coeff_Czu}
[z^n]C(z,u)\sim\frac{\kappa(u)}{\Gamma(s)}\,\rho(u)^{-n}\,n^{s-1}
\end{equation}
uniformly for $u$ in a neighborhood of $1$, with $\kappa(u)$ analytic at $1$.  
The conditions of the quasi-power theorem in its local limit formulation~\cite[Theo.IX.14]{flajolet2009analytic} are satisfied: defining 
\[
\mathfrak{e}=-\frac{\rho'(1)}{\rho},\ \ \ \mathfrak{s}^2=-\frac{\rho''(1)}{\rho}-\frac{\rho'(1)}{\rho}+\frac{\rho'(1)^2}{\rho^2},
\]
we have the Gaussian estimate
\[
\mathbb{P}\Big(X_n=\mathfrak{e}\,n+t\,\mathfrak{s}\,\sqrt{n}\Big)\sim \frac{1}{\sqrt{n}}\frac1{\mathfrak{s}\sqrt{2\pi}}e^{-t^2/2}
\]
uniformly for $t$ in any fixed compact set. Moreover, from the equation $B(\rho(u))=\rho_A/u$ it is easy to check that $\mathfrak{e}=1/\mu$ and $\mathfrak{s}^2=\sigma^2/\mu^3$. 

Regarding the tail bound, we can establish it classically by a Chernoff bound, with calculations similar to those for the sum of independent random variables, see e.g.~\cite[Theo.15]{petrov2012sums}. From the estimate~\eqref{eq:coeff_Czu} there exists a real neighborhood $U_1$ of $1$ and a constant $\tilde{\xi}$ such that, for all $n\geq 0$ and $u\in U_1$,
\[
p_n(u)\leq \tilde{\xi}\, \big(\rho(u)/\rho\big)^{-n}.
\] 
Since $p_n(u)=\mathbb{E}(e^{sX_n})$ for $u=e^s$, there is a real neighborhood $U_0$ of $0$ such that, for all $n\geq 0$ and $s\in U_0$,
\[
\mathbb{E}\big(e^{s(X_n-n/\mu)}\big)\leq \tilde{\xi}\,\exp(n\,g(s)),\ \ \mathrm{with}\ g(s)=-\log(\rho(e^s))+\log(\rho)-s/\mu.
\]
The function $g(s)$ satisfies $g(0)=g'(0)=0$, so that in a real neighborhood $[-\tilde{\zeta},\tilde{\zeta}]$ of $0$  we have $g(s)\leq \tilde{\kappa} s^2$ for some $\tilde{\kappa}>0$. 
Thus, for $n\geq 0$ and $s\in[-\tilde{\zeta}\sqrt{n},\tilde{\zeta}\sqrt{n}]$, 
\[
\mathbb{E}(e^{s(X_n-n/\mu)/\sqrt{n}})\leq \tilde{\xi}\,\exp(\tilde{\kappa} s^2)
\]
and in that case Markov's inequality gives 
\[
\mathbb{P}\big(|X_n-n/\mu|\geq t\sqrt{n}\big)\leq 2\,\tilde{\xi}\,\exp(\tilde{\kappa} s^2-ts)
\]
for any $t>0$. If $t\leq 2\tilde{\kappa}\,\tilde{\zeta}\,\sqrt{n}$, then the optimal value $s=\frac{t}{2\tilde{\kappa}}$
 satisfies $s\leq \tilde{\zeta}\sqrt{n}$, so we have
 \[
\mathbb{P}\big(|X_n-n/\mu|\geq t\sqrt{n}\big)\leq 2\,\tilde{\xi}\,\exp(-t^2/(4\tilde{\kappa})).
\]
The tail bound estimate follows, taking $\xi=2\tilde{\xi}$, $\kappa=\frac{1}{4\tilde{\kappa}}\frac{\sigma^2}{\mu^3}$, and $\zeta=2\tilde{\kappa}\,\tilde{\zeta}\frac{\sqrt{\mu^3}}{\sigma}$.
\end{proof}

\begin{lem}\label{lem:estimate_dnk}
Let $\epsilon_n=o(\sqrt{n})$. 
For $k=\frac{n}{\mu}+t\frac{\sigma}{\mu^{3/2}}\sqrt{n}$ such that $|t|\leq \epsilon_n$, one has
\begin{equation}\label{eq:estimate_dnk}
1-d_{n,k}=\frac{\sigma}{\sqrt{n\mu}}(s-1)\,t+O((1+t^2)/n).
\end{equation}
\end{lem}
\begin{proof}
We have, for $k=\Theta(n)$,
\[
\frac{c_n\rho^n}{a_kB(\rho)^k}=\frac{\kappa_C\, n^{s-1}\,\big(1+O(1/n)\big)}{\kappa_A\, k^{s-1}\,\big(1+O(1/k)\big)}=\frac{\kappa_C}{\kappa_A}\Big(\frac{k}{n}\Big)^{1-s}\,\big(1+O(1/n)\big).
\]
For $\frac{k}{n}=\frac1{\mu}\big(1+t\frac{\sigma}{\sqrt{n\mu}}\big)$ with $|t|\leq \epsilon_n$, this gives, using $\frac{\kappa_C}{\kappa_A}=\mu^{-s}$ (as seen in~\eqref{eq:kappa}),
\[
\frac{c_n\rho^n}{a_kB(\rho)^k}=\frac{1}{\mu}\Big(1+\frac{\sigma(1-s)t}{\sqrt{n\mu}}\Big)+O((1+t^2)/n).
\]
Combined with~\eqref{eq:estimate_qn}, this yields~\eqref{eq:estimate_dnk}.
\end{proof}

\begin{prop}\label{prop:asympt_uniform}
Let $\cC=\cA\circ\cB$ be an admissible composition scheme of singular exponent $s$.  
Then
\begin{equation}\label{eq:dtvasympt}
\dtv(\pi_n,\pi_n')\sim\frac{|s-1|\sigma}{\sqrt{2\pi\mu}}\frac1{\sqrt{n}}.
\end{equation}
\end{prop}
\begin{proof}
For $y>0$ let  
\[
\Int_n(y):=\Big[n/\mu-y\,\sigma\sqrt{n/\mu^3},\,n/\mu+y\,\sigma\sqrt{n/\mu^3}\Big]\cap\Big[1..n\Big].
\]  
Let $I_n:=\Int_n(\log(n)^2)$, and $J_n:=[1..n]\backslash I_n$, so we have
\[
\dtv(\pi_n,\pi_n')=\frac1{2}\sum_{k\in I_n}S_{n,k}+\frac1{2}\sum_{k\in J_n}S_{n,k}.
\]  
We first show that $\sum_{k\in J_n}S_{n,k}$ is asymptotically negligible.
By~\eqref{eq:large_dev} and the fact that $d_{n,k}=O(n^{(s-1)_+})$, we have
\[
\sum_{k\in J_n}S_{n,k}=O\Big(n^{(s-1)_+}\sum_{k\in J_n}\frac{c_{n,k}}{c_n}\Big)
=O\big(n^{|s-1|}e^{-\kappa\log(n)^2}\big)=O(n^{-b})\ \mathrm{for\ any\ }b>0.
\]

We now consider the contribution $\sum_{k\in I_n}S_{n,k}$. 
Let $m>0$. Then from~\eqref{eq:large_dev} and~\eqref{eq:estimate_dnk} we obtain, for $k=\frac{n}{\mu}+t\,\sigma\sqrt{n/\mu^3}$,
\[
S_{n,k}\sim \frac{\mu\,|s-1|}{\sqrt{2\pi}\,n}|t|\,e^{-t^2/2},
\]
uniformly over $k\in I_n(m)$, so that  
\begin{align*}
\sum_{k\in I_n(m)}S_{n,k}\sim \frac{\mu\,|s-1|}{\sqrt{2\pi}\,n}\sum_{k\in I_n(m)}|t|\,e^{-t^2/2}& \sim \frac{\mu\,|s-1|}{\sqrt{2\pi}\,n}\frac{\sigma\sqrt{n}}{\sqrt{\mu^3}}\,\int_{-m}^m |t|\,e^{-t^2/2}\,\mathrm{d}t\\
& = \frac{\sigma\,|s-1|}{\sqrt{2\pi\,\mu\,n}}\,\int_{-m}^{m} |t|\,e^{-t^2/2}\,\mathrm{d}t
\end{align*}

On the other hand, for $0\leq m\leq \log(n)^2$, by Lemma~\ref{lem:estimate_cnk} and Lemma~\ref{lem:estimate_dnk} we have the existence of universal constants such that
\[
\sum_{k\notin I_n(m)}\frac{c_{n,k}}{c_n}=O(e^{-\kappa m^2}),\ \ \ |1-d_{n,k}|=O\big(m/\sqrt{n}\big)\ \mathrm{for}\ k\in I_n(m)\backslash I_n(m+1)
\]
Hence, for $0\leq m\leq \log(n)^2$, we have
\[
\sum_{k\in I_n(m+1)\backslash I_n(m)}S_{n,k}=\frac1{\sqrt{n}}O(m\,e^{-\kappa m^2}),
\]
By summation we obtain 
\[
\sum_{k\in I_n\backslash I_n(m)}S_{n,k}=\frac1{\sqrt{n}}O(m\,e^{-\kappa m^2})
\]
Hence, for every $m\geq 1$
\[
\sum_{k\in I_n}S_{n,k}=(1+o(1))\frac1{\sqrt{n}} \frac{\sigma\,|s-1|}{\sqrt{2\pi\,\mu}}\,\int_{-m}^{m} |t|\,e^{-t^2/2}\,\mathrm{d}t + \frac1{\sqrt{n}}O(m\,e^{-\kappa m^2}),
\]
where the little $o$ depends on $m$, but not the big $O$. 
Letting $m$ get large, we obtain

\[
\sum_{k\in I_n}S_{n,k}\sim\frac1{\sqrt{n}} \frac{\sigma\,|s-1|}{\sqrt{2\pi\,\mu}}\,\int_{\mathbb{R}} |t|\,e^{-t^2/2}\,\mathrm{d}t=2\frac{|s-1|\sigma}{\sqrt{2\pi\mu}}\frac1{\sqrt{n}}.
\]
\end{proof}

\begin{remark}
Note that the value $s=1$ plays a special role, it is indeed the unique singular exponent where $\dtv(\pi_n,\pi_n')$ is not $\Theta(1/\sqrt{n})$ but smaller. This is consistent with Remark~\ref{rk:pinp_coresize}, and with the fact that $\dtv(\pi_n,\pi_n')=0$ when $\cA=\mathrm{Seq}$ (a case where $s=1$ since $A(y)=\frac{1}{1-y}$)
 as covered in~\cite[Section 7.1]{duchon2004boltzmann}.   
\end{remark}

\begin{remark}
If we only have less precise estimates $[y^k]A(y)\sim \kappa_A\,\rho_A^{-k}\,k^{s-1}$ and 
$[z^n]C(z)\sim \kappa_C\,\rho_C^{-n}\,n^{s-1}$, we still have asymptotic uniformity. 
Indeed, for $k\in I_n$, 
\[
d_{n,k}=\frac1{q_n}\frac{c_n\rho^n}{a_kB(\rho)^k}\sim\mu\frac{\kappa_C\, n^{s-1}}{\kappa_A\,k^{s-1}}\sim\,\mu^s\frac{\kappa_C}{\kappa_A}\sim 1,
\]
so that 
$\sum_{k\in I_n}S_{n,k}=\sum_{k\in I_n}\frac{c_{n,k}}{c_n}|1-(1+o(1)|=o(1)$.

\end{remark}

\subsection{Acceleration of convergence}
We show here that a simple rejection step on top of leap generators makes it possible to accelerate the rate of convergence toward the uniform distribution. 
A \emph{rejection sequence} is given by $w_{n,k}\in[0,1]$, for $0\leq k\leq n$.  
Given such a sequence, and an admissible composition scheme $\cC=\cA\circ\cB$, the \emph{rejection leap generator} for $\cC$ proceeds as follows:

\medskip

\begin{tabular}{ll}
$\Gammal C^{\mathrm{\,rej}}[n]$:& repeat \\
&\hspace{.4cm} $\gamma\leftarrow\Gammal C[n]$ \\
&\hspace{.4cm} let $k=$ core-size of $\gamma$\\
& until $\mathrm{Bern}(w_{n,k})$\\
& return $\gamma$
\end{tabular}

\medskip
\medskip

\begin{lem}\label{lem:estimate_dnk_rej}
Let $W_n:=\sum_{k}\frac{c_{n,k}}{c_n}d_{n,k}w_{n,k}$, and  let $\pi_n^{\mathrm{rej}}$ be the distribution on $\cC_n$ given by $\Gammal C^{\mathrm{\,rej}}[n]$. 

Then, for $\gamma\in\cC_{n,k}$, we have
\[
\pi_n^{\mathrm{rej}}(\gamma)=\frac1{\cC_n}d_{n,k}^{\mathrm{rej}}\ \ \ \ \mathrm{with}\ d_{n,k}^{\mathrm{rej}}:=\frac1{W_n}d_{n,k}w_{n,k}.
\]
Moreover,
\[
\dtv(\pi_n,\pi_n^{\mathrm{rej}})=\frac1{2}\sum_{k=0}^n\frac{c_{n,k}}{c_n}\big|1-d_{n,k}^{\mathrm{rej}}\big|.
\]
\end{lem}
\begin{proof}
For $\gamma\in\cC_{n}$ of core-size $k$, let $\pi_n^{\mathrm{trial}}(\gamma)$ be the probability of drawing $\gamma$ at a given trial. 
Then
\[
\pi_n^{\mathrm{trial}}(\gamma)=\pi_n'(\gamma)w_{n,k}.
\]
The probability of success at each trial, denoted $W_n$, is thus the sum of these probabilities over all $\gamma\in\cC_n$, so that 
\[
W_n=\sum_{\gamma\in\cC_n} \pi_n^{\mathrm{trial}}(\gamma)=\sum_{k=0}^n\frac{|\cC_{n,k}|}{|\cC_n|}d_{n,k}w_{n,k}=\sum_{k=0}^n\frac{c_{n,k}}{c_n}d_{n,k}w_{n,k}.
\]
Then we have, for $\gamma\in\cC_n$,
\[
\pi_n^{\mathrm{rej}}(\gamma)=\frac1{W_n}\pi_n^{\mathrm{trial}}(\gamma),
\]
which gives the claimed formula.
\end{proof}

Our aim here is to exhibit easy-to-compute rejection sequences giving asymptotic uniformity at a faster rate of convergence than in Proposition~\ref{prop:asympt_uniform}. We say that  a generating function $F(z)$ of singular exponent $s$ has a \emph{complete asymptotic expansion} 
if there exist coefficients $f_\infty^{1},f_\infty^{2},\ldots$ such that, for every $r\geq 0$,
\[
[z^n]F(z)=\frac{\kappa_F}{\Gamma(s)}\rho^{-n}n^{s-1}\Big(1+\sum_{i=1}^{r}f_\infty^{i}n^{-i}+O\big(1/n^{r+1}\big)\Big)
\]
Let $\cC=\cA\circ\cB$ be an admissible composition scheme such that $A(y)$ has a complete asymptotic expansion (an assumption to be satisfied by all examples to be given in Section~\ref{sec:appli}). Then $C(z)$ also has a complete asymptotic expansion, whose coefficients are obtained from those of $A(y)$ and from the derivatives of $B(z)$ at $\rho$. Then we have the following generalization of Lemma~\ref{lem:estimate_dnk}.

\begin{lem}\label{lem:estimate_dnk_general}
Let $\cC=\cA\circ\cB$ be an admissible composition scheme of singular exponent~$s$, and such that $A(y)$ has a complete asymptotic expansion. Let $\epsilon_n=o(\sqrt{n})$. Then there exist explicit polynomials $p_\infty^1(t),p_\infty^2(t),\ldots$, with $p_\infty^i(t)$ of degree at most~$i$, and having only even (resp. odd) powers of $t$ for $i$ even (resp. odd), such that, for $k=\frac{n}{\mu}\Big(1+t\frac{\sigma}{\sqrt{\mu n}}\Big)$ with $|t|\leq \epsilon_n$, one has
\begin{equation}\label{eq:asympt_dnk}
d_{n,k}=1+\sum_{i=1}^r p_\infty^i(t)\,n^{-i/2}+O\left(\Big(\frac{1+t^2}{n}\Big)^{\frac{r+1}{2}}\right)
\end{equation}
for any fixed $r\geq 0$. 

The first polynomials are, with $a_\infty^i$ the coefficients in the asymptotic expansion of $A(y)$, and $c_\infty^i$ those of $C(z)$: 
\[
p_\infty^1(t)=-\frac{\sigma}{\sqrt{\mu}} (s-1)t,\ \ \ 
p_\infty^2(t)=\frac{(s-1)s\sigma^2}{2\mu}t^2+c_\infty^1-a_\infty^1\mu,
\]
\[
p_\infty^3(t)=\frac{s(1-s^2)\sigma^3}{6\,\mu^{3/2}}t^3+\Big(\sigma\sqrt{\mu}\,s\,a_\infty^1 -(s\!-\!1)\frac{\sigma}{\sqrt{\mu}}\,c_\infty^1\Big)\,t.
\]
\end{lem}
\begin{proof}
Let $m=\lfloor r/2\rfloor$. We have, for $k=\Theta(n)$,
\[
\frac{c_n\rho^n}{a_kB(\rho)^k}=
\frac{\kappa_C}{\kappa_A}\Big(\frac{k}{n}\Big)^{1-s}\Big(1+\sum_{i=1}^mc_\infty^i n^{-i}+O(n^{-m-1})\Big)\Big(1+\sum_{i=1}^ma_\infty^i k^{-i}+O(k^{-m-1})\Big)^{-1}
\]
Injecting $k=\frac{n}{\mu}\big(1+t\frac{\sigma}{\sqrt{n\mu}}\big)$ and expanding according to $n$ yields~\eqref{eq:asympt_dnk}, valid for $|t|\leq\epsilon_n$. 
\end{proof}

This leads to the following generalization of Proposition~\ref{prop:asympt_uniform}. 
\begin{prop}\label{prop:asympt_uniform_general}
Let $\cC=\cA\circ\cB$ be an admissible composition scheme of singular exponent $s$, and such that $A(y)$ has a complete asymptotic expansion. 

Let $a\in(0,1)$ (e.g. $a=1/2$) and let $r\in\mathbb{N}$. With the notation of Lemma~\ref{lem:estimate_dnk_general}, let $y_n^{(r)}(t)=1+\sum_{i=1}^r p_\infty^i(t)\,n^{-i/2}$. 
For $0\leq k\leq n$, with $t:=\frac{k-n/\mu}{\sigma\sqrt{n/\mu^3}}$, we define a rejection sequence $w_{n,k}$ as follows:
\[
\mathrm{if}\ y_n^{(r)}(t)\geq a\ \ \mathrm{then\ set}\ w_{n,k}:=\frac{a}{y_n^{(r)}(t)},\ \ \ \mathrm{otherwise}\ \ \mathrm{set}\ w_{n,k}=1.
\]
Then with this rejection sequence, we have
\begin{equation}
\dtv(\pi_n,\pi_n^{\mathrm{rej}})\sim \frac{c}{2\,n^{\frac{r+1}{2}}}
\end{equation}
where
\[
c=\frac{1}{\sqrt{2\pi}}\int_{\mathbb{R}}e^{-t^2/2}\,\big| p_\infty^{r+1}(t) -I\big|\,\mathrm{d}t,\ \ \mathrm{with}\ I:=\frac{1}{\sqrt{2\pi}}\int_{\mathbb{R}}e^{-t^2/2}\, p_\infty^{r+1}(t)\,\mathrm{d}t.
\]
The probability of success of each trial tends to $a$ as $n\to\infty$, so that the expected complexity of $\Gammal C^{\mathrm{\,rej}} [n]$ is $O(n)$.     
\end{prop}
\begin{proof}
With the notation of Lemma~\ref{lem:estimate_dnk_rej}, we aim first at obtaining an asymptotic estimate for $W_n:=\sum_k\frac{c_{n,k}}{c_n}w_{n,k}d_{n,k}$. Let $\epsilon_n=o(\sqrt{n})$. Then, for $k=\frac{n}{\mu}\Big(1+t\frac{\sigma}{\sqrt{\mu n}}\Big)$ with $|t|\leq \epsilon_n$, by Lemma~\ref{lem:estimate_dnk_general} we have 
\[
w_{n,k}d_{n,k}=a+a\,p_\infty^{r+1}(t)n^{-(r+1)/2}+O\left(\Big(\frac{1+t^2}{n}\Big)^{\frac{r+2}{2}}\right).
\]
Hence,
\[
W_n-a=\sum_k\frac{c_{n,k}}{c_n}\big(w_{n,k}d_{n,k}-a\big)\sim \frac{a}{n^{\frac{r+1}{2}}}I,
\]
where the equivalence of the sum to the integral can be established by the exact same argument (integrating over $[-m,m]$ and letting $m\to\infty$) as in the proof of Proposition~\ref{prop:asympt_uniform}. 

Hence, uniformly for $|t|\leq \epsilon_n$, we have
\begin{align}
d_{n,k}^{\mathrm{rej}}&:=\frac1{W_n}w_{n,k}d_{n,k}\nonumber\\
&=
1+\big(p_\infty^{r+1}(t)-I\big)n^{-\frac{r+1}{2}}+o\big(n^{-\frac{r+1}{2}}\big)+O\left(\Big(\frac{1+t^2}{n}\Big)^{\frac{r+2}{2}}\right)\label{eq:dnkrej}
\end{align}
from which we obtain (again by the same argument as in the proof of Proposition~\ref{prop:asympt_uniform})
\[
\sum_k \frac{c_{n,k}}{c_n}\Big|1- d_{n,k}^{\mathrm{rej}}\Big|\sim \frac{c}{n^{\frac{r+1}{2}}}.
\]
\end{proof}


\begin{remark}
For $r=0$ we have $y_n^{(0)}(t)=1$, so that $w_{n,k}=a$. Hence $\pi_n^{\mathrm{rej}}=\pi_n'$ whatever the chosen value of $a\in(0,1)$, and taking the limit value $a=1$, the generator $\Gammal C^{\mathrm{\,rej}} [n]$ is exactly $\Gammal C[n]$. 
\end{remark}

\subsection{Extensions}
As briefly described below, the leap generation method and the (possibly accelerated) asymptotic uniformity work in more general settings. 

\subsubsection{Composition schemes with a distinguished first component} 
One can consider more generally a composition scheme of the form
\[
\cC=\cD\times\, \cA\circ\cB.
\]
Such a scheme is called \emph{supercritical} if $\rho_B>\rho_C$ and $\rho_D>\rho_C$. 
It is called \emph{admissible} if furthermore $B(z),C(z)$ are aperiodic and the support of $A(y)$ is $\mathbb{N}$. In that case (with $\rho:=\rho_C$) we let $\Gammal C[n]$ generate a sequence $\delta\leftarrow \Gamma D(\rho),\beta_1\leftarrow\Gamma B(\rho),\ldots,\beta_k\leftarrow\Gamma B(\rho)$ until the total size $S:=|\delta|+|\beta_1|+\cdots+|\beta_k|$ satisfies $S\geq n$ (possibly $k=0$) and we reject and restart if $S>n$. If $S=n$ we call $\alpha\leftarrow \Gamma A[k]$ and return the pair $\gamma$ made of $\delta$ and of the composition of $\alpha$ with $\beta_1,\ldots,\beta_k$ (with a final relabeling operation if in the labeled setting). The formulas for $q_n$ and $d_{n,k}$  are slightly modified to
\begin{equation}\label{eq:case_with_first_component}
q_n=[z^n]\frac{D(\rho z)}{D(\rho)}\frac{1}{1-B(\rho z)/B(\rho)},\ \ \ d_{n,k}=\frac1{q_n}\frac{c_n\rho^n}{D(\rho)\,a_k\,B(\rho)^k},
\end{equation}
while the formula for the total variation distance remains 
$\dtv(\pi_n,\pi_n')=\frac1{2}\sum_k S_{n,k}$, with $S_{n,k}=\frac{c_{n,k}}{c_n}|1-d_{n,k}|$.  
The estimates~\eqref{eq:estimate_qn} and~\eqref{eq:dtvasympt} are still valid, as well as the linear running time of $\Gammal C[n]$, assuming that the samplers $\Gamma A[k],\Gamma B(x),\Gamma D(x)$ have running time linearly bounded by the size of the output.

\subsubsection{Periodic case} 
If the support of $A(y)$ is not $\mathbb{N}$, then each trial of sequence generation $(\delta,\beta_1,\ldots,\beta_k)$ is declared a success if $S=n$ \emph{and} $k$ is in the support of~$A(y)$. As before, once a successful trial is obtained, one calls $\alpha\leftarrow \Gamma A[k]$ and returns the pair made of $\delta$ and of $\alpha$ composed with $\beta_1,\ldots,\beta_k$.  
This situation with gaps in the support of $A(y)$ typically has to be considered when $A(y)$ is periodic. 
We take the example where the support of $D(z)$ is $d\,\mathbb{N}$ and the supports of $A(y),B(z),C(z)$ are $1+d\,\mathbb{N}$ for some fixed $d\geq 2$. Then the formulas for $q_n$ and $d_{n,k}$ are still given by~\eqref{eq:case_with_first_component} 
under the condition that $n,k\in 1+d\,\mathbb{N}$. And in the formula for $\dtv(\pi_n,\pi_n')$ the sum is restricted to $k\in 1+d\,\mathbb{N}$. 
The estimate~\eqref{eq:estimate_qn} is still valid upon conditioning on $n\in 1+d\mathbb{N}$ 
and the sampler $\Gammal C[n]$ still has expected running time $O(n)$ if the samplers $\Gamma A[k],\Gamma B(x), \Gamma D(x)$ have running time linearly bounded by the size of the output. 
The estimate~\eqref{eq:dtvasympt} is also still valid, with singular exponent defined in the $d$-periodic setting 
(again conditioning on $n\in 1+d\mathbb{N}$).  

\subsubsection{Weighted case} The method also extends to weighted structures. In a \emph{weighted class} $\cC$, to each structure $\gamma\in\cC$ is associated a weight $w(\gamma)\in\mathbb{R}_+$. In that case, the associated generating function and the $w$-distribution 
$\pi_n$ on $\cC_n$ are given by 
\[
C(z)=\sum_{\gamma\in\cC}\delta_{|\gamma|}w(\gamma)z^{|\gamma|},\ \ \ \pi_n(\gamma)=\frac{\delta_nw(\gamma)}{[z^n]C(z)}\ \ \mathrm{for}\ \gamma\in\cC_n.
\]  
The weights are usually required to follow multiplicative rules when assembling structures. Thus, 
for a composition scheme $\cC=\cD\times\cA\circ\cB$, and for $\gamma\in \cC$ made of $\delta\in\cD$ and of $\alpha\in\cA$ composed with $\beta_1,\ldots,\beta_k$, we require 
$w(\gamma)=w(\delta)w(\alpha)\prod_{i=1}^kw(\beta_i)$, so that the generating functions are related by $C(z)=D(z)\,A(B(z))$. The definitions of supercritical and admissible composition schemes extend straightforwardly, as well as the description of the leap generator $\Gammal C[n]$, assuming 
 the generator $\Gamma A[k]$ draws a structure in $\cA_k$ under the $w$-distribution, while the generator $\Gamma B(x)$ (and similarly $\Gamma D(x)$) draws a structure in $\cB$  under the $w$-weighted Boltzmann distribution $\mathbb{P}_x^\cB(\beta)=\delta_{|\beta|}w(\beta)x^{|\beta|}/B(x)$. As before, the sampler $\Gammal C[n]$ has expected running time $O(n)$ if the samplers for the component classes have running time linearly bounded by the size of the drawn object. 
 Letting  $\pi_n'$ be the induced distribution on $\cC_n$, the estimate~\eqref{eq:dtvasympt} remains valid, so that $\pi_n'$ is asymptotically the $w$-distribution $\pi_n$, in the total variation sense.

\subsubsection{Critical composition schemes} Finally, the method can also be extended to certain \emph{critical} composition schemes. In such a scheme $\cC=\cD\times\cA\circ\cB$ we have $\rho_B=\rho_C$ and $B(\rho_B)=\rho_A$ (and $\rho_D\geq \rho_C$). The leap generator $\Gammal C[n]$ can be defined exactly as in the supercritical case. While we do not carry a general analysis, we will see in Section~\ref{sec:maps} an important example (and extensions) of such a scheme: the block decomposition of planar maps. We will show that the expected linear time complexity and the asymptotic uniformity (up to using a rerooting operation on top of $\Gammal C[n]$) are still satisfied for such a scheme, the order of magnitude of $\dtv(\pi_n,\pi_n')$ being of order $n^{-1/3}$ instead of $n^{-1/2}$.

\section{Applications}\label{sec:appli}
We describe here how the strategy applies to various classes, of walks, trees and planar maps. In each case we take advantage of a decomposition in terms of a core-structure for which  exact-size uniform random sampling can be done in linear time. 
\subsection{Walks}
We first show how leap generators can be obtained for Motzkin and Schr\"oder walks. 
For these families, exact-size uniform random samplers of linear time complexity have been developed~\cite{Alonso94,gouyou2010random,bacher2013exact,bacher2018improving}, so that applying leap generation is of limited interest. We rather consider these as toy examples where the strategy is simplest to implement and to validate experimentally (see Section~\ref{sec:experiments}).  
\subsubsection{Motzkin walks}
A \emph{Motzkin walk} of length $n\geq 0$ is a walk with $n$ steps in $\{NE,E,SE\}$, starting at the origin, ending at $(n,0)$, and staying in the upper half-plane $\{y\geq 0\}$. 
\emph{Dyck walks} are Motzkin walks with no horizontal step. Let $\cC$ be the class of Motzkin  walks where the size is the length, and $\cA$ the class of Dyck walks where the size is the semi-length. A Motzkin walk is uniquely obtained from a Dyck walk where each point can be extended into a (possibly empty) sequence of horizontal steps. Since a Dyck walk of semi-length $k$ has $2k+1$ points, this yields
\[
\cC=\sum_{k\geq 0}\cA_k\cZ^{2k}\Seq(\cZ)^{2k+1}.
\]
Hence, letting $\cB=(\cZ\times \Seq(\cZ))^2$ and $\cD=\Seq(\cZ)$, we have
\[
\cC=\cD\times\cA\circ\cB.
\]
The composition scheme is admissible, with $\rho_B=\rho_D=1>\rho_C=1/3$, and the class $\cA$ has simple exact-size uniform random samplers (e.g. via the cyclic lemma). In the leap generation process, each call to the Boltzmann sampler $\Gamma B(1/3)$ returns 
the pair $(i,j)$ (considered to have size $i+j+2$), 
with $i,j$ drawn independently under a $\mathrm{Geom}(1/3)$ law.    

\subsubsection{Schr\"oder walks}
A \emph{Schr\"oder walk} of size $n\geq 0$ is a walk with steps in $\{N,E,NE\}$,
starting at the origin, ending at $(n,n)$, and staying in $\{y\geq x\}$. 
Upon rotating by $45$ degrees, these correspond to Motzkin walks whose length plus number of horizontal steps is equal to $2n$. Similarly as before, each Schr\"oder walk is obtained as a Dyck walk with a possibly empty sequence of horizontal steps inserted at each point (compared to the Motzkin case, only the size parameter behaves differently). Letting $\cC$ be the class of Schr\"oder excursions, and $\cA$ the class of Dyck walks where the size is the semi-length, this gives
\[
\cC=\sum_{k\geq 0}\cA_k\cZ^{k}\Seq(\cZ)^{2k+1}
\]
Hence, letting $\cB=\cZ\times \Seq(\cZ)^2$ and $\cD=\Seq(\cZ)$, we have
\[
\cC=\cD\times\cA\circ\cB.
\]
The composition scheme is again admissible, with $\rho_B=\rho_D=1>\rho_C=3-2\sqrt{2}$. In the leap generation process, each call to the Boltzmann sampler $\Gamma B(\rho_C)$ returns 
the pair $(i,j)$ (considered to have size $i+j+1$), 
with $i,j$ drawn independently under a $\mathrm{Geom}(\rho_C)$ law.  

\subsection{Trees}
Leap generators can also be obtained for various classes of unlabeled rooted trees, a main motivation for this work.
\subsubsection{P\'olya trees}
Recall that a \emph{P\'olya tree} (resp. a rooted Cayley tree) is a rooted unlabeled (resp. labeled) tree where the children of nodes are not ordered, the size being the number of vertices. 
The class $\cA$ of rooted Cayley trees is given by 
\[
\cA=\cZ\times\Set(\cA),\ \ A(z)=z\exp(A(z)),\ \ |\cA_n|=n^{n-1},
\]
and one can perform uniform random generation in $\cA_n$ in linear time (e.g. via Pr\"ufer code). 
The class of P\'olya trees is the unlabeled class $\tcA$ associated to $\cA$, it is given by
\begin{equation}\label{eq:tcA}
\tcA=\cZ\times\MSet(\tcA),\ \ \tA(z)=z\exp\big(\sum_{i\geq 1}\frac1{i}\tA(z^i)\big).
\end{equation}
Letting $\cC=\Sym(\cA)$, we have 
\begin{equation}\label{eq:PolyaCA}
\cC_n\simeq n!\times \tcA_n,
\end{equation}
 so that asymptotically uniform random sampling in $\tcA_n$ reduces to asymptotically uniform random sampling in $\cC_n$, as recalled in Section~\ref{sec:background_tools}. Let $(\sigma,\gamma)\in\cC_n$, with $\gamma\in\cA_n$ and $\sigma$ an automorphism of $\gamma$. 
Following~\cite{panagiotou2018scaling}, one can consider the \emph{core} of $(\sigma,\gamma)$ as the subtree (containing the root-vertex) formed by the vertices that are fixed by $\sigma$. 
\begin{figure}
\begin{center}
\includegraphics[width=12cm]{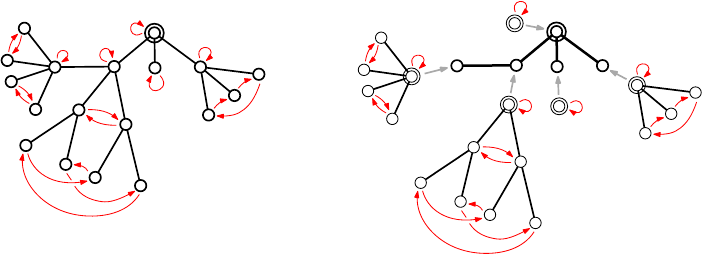}
\end{center}
\caption{Left: a rooted tree (labels not indicated) endowed with an automorphism. Right:   decomposition into a core (the subtree made of the fixed vertices) where at each vertex is attached a rooted tree endowed with an automorphism fixing only its root-vertex.}
\label{fig:polya}
\end{figure}
For each vertex $v$ of the core, let $\tau_v$ be  the subtree made of $v$ and the subtrees hanging from the unfixed children (if any) of $v$. 
Letting $\cB$ be the subclass of $\cC$ such that the root-vertex is the only fixed vertex, one thus has (see Figure~\ref{fig:polya})
\[
\cC=\cA\circ\cB.
\]   
This is reflected at the level of generating functions. Indeed, the equation for $\tA(z)$ is equivalent to 
\[
\tA(z)=\Big(z\exp\big(\sum_{i\geq 2}\frac1{i}\tA(z^i)\big)\Big)\exp(\tA(z)),
\]
hence one has $\tA(z)=A(B(z))$, where $B(z)=z\exp\big(\sum_{i\geq 2}\frac1{i}\tA(z^i)\big)$. 
Moreover~\eqref{eq:PolyaCA} ensures that $\tA(z)$ is also the exponential generating function $C(z)$ of the class $\cC$. And P\'olya theory ensures that $B(z)$ is the exponential generating function of $\cB$; indeed the cycle index sum of sets endowed with a permutation with no fixed point is $\exp(\sum_{i\geq 2}\frac1{i}s_i)$.   

This composition scheme is known to be admissible, with $\rho_C\approx 0.338$ and ${\rho_B=\rho_C^{1/2}}$. 
The Boltzmann sampler $\Gamma B(x)$ (to be called at $\rho_C$ in the leap generation process) is deduced from the known (ordinary) Boltzmann sampler $\Gamma \tA(x)$ for P\'olya trees~\cite{flajolet2007boltzmann}. It is shown in Figure~\ref{algo:gamma-B} and, like $\Gamma \tA(x)$, has linear complexity in the size of the output. On the other hand,  an exact-size sampler  $\Gamma A[k]$ is readily obtained, for instance via the Pr\"ufer encoding of Cayley trees. Thus the leap generator for $\cC$ also has linear time complexity. 

\begin{remark}
The very recent generator developed in~\cite{stuflerPolya} achieves complexity $O(n)$ for uniform exact-size sampling of P\'olya trees. It also builds on a leap generation process involving $\Gamma B(\rho_C)$, but it tilts the acceptance probability of each structure (depending on $n$ and the core-size $k$) so that the obtained distribution on $\cC_n$ is uniform, while managing to boost these acceptance probabilities (following ideas developed in~\cite{sportiello}) so that the probability of success of each attempt is $\Omega(1)$. It is however a bit more involved than our leap generator, as the acceptance probability needs to be computed at each attempt, requiring the precomputation of $\Theta(n)$  numerical estimates of real numbers (such as the log of factorial numbers). 
\end{remark}

\begin{figure}\small
\fbox{\begin{minipage}{.82\textwidth}
Let $\textsc{Max\_Index}(x)$ be a  generator according to the following distribution:
\begin{equation*}\label{defmax}
\Pr(K\leq k)=\frac{x}{\tilde{A}(x)}\prod_{j\leq k}\exp\left(\frac{1}{j}\tilde{A}(x^j)\right)
.
\end{equation*}

\smallskip

\begin{tabbing}
XXX \= XX \= XXX \= XXX \= XXX \kill
{\bf Algorithm } \textsc{Random\_Forest}$(x)$~:\\
\> $\gamma\leftarrow \varnothing$; $k_0\leftarrow \textsc{Max\_Index}(x)$; \\
\>{\bf if} $k_0 < 2$ {\bf then~return} $\gamma$\\
\> {\bf for} $j$ {\bf from} $2$ {\bf to} $k_0-1$ {\bf do}\\
\>\> $\gamma \leftarrow \gamma,\left[\operatorname{Pois}\left( \frac{\tilde{A}(x^j)}{j}\right)
\implies \operatorname{copy}(j,\Gamma \tilde{A}(x^j))\right]$\\
\> $\gamma \leftarrow \gamma, \left[\operatorname{Pois}_{\geq 1}\left( \frac{\tilde{A}(x^{k_0})}{k_0}\right)
\implies \operatorname{copy}(k_0,\Gamma \tilde{A}(x^{k_0}))\right]$\\
\> {\bf return} $\gamma$
\end{tabbing}

\smallskip

\begin{tabbing}
XXX \= XX \= XXX \= XXX \= XXX \kill
{\bf Algorithm } $\Gamma B(x)$~:\\
\> $(\gamma_1,\dotsc,\gamma_m) \leftarrow$ \textsc{Random\_Forest}$(x)$\\
\> {\bf return} $[\mathcal Z, [\gamma_1,\dotsc,\gamma_m]]$
\end{tabbing}
\end{minipage}}

\caption{\label{algo:gamma-B}
Algorithm $\Gamma B(x)$ for P\'olya trees}
\end{figure}

\subsubsection{Unlabeled phylogenetic trees}
A phylogenetic tree is a rooted binary tree such that the two children at each node are unordered. For convenience, the tree with a single leaf is also included in the class; the size $n$ is the number of leaves. 
Let $\cA$ be the class of labeled (at leaves) phylogenetic trees, given by
\[
\cA=\cZ+\Set_2(\cA),\ \ \ A(z)=z+\frac1{2}A(z)^2,\ \ \ |\cA_n|=(2n-3)!!,
\]
for which uniform random sampling at size $n$ can be done in linear time (e.g. by Pr\"ufer codes, or R\'emy's algorithm). 
Let $\tcA$ be the associated unlabeled class, given by
\[
\tcA=\cZ+\MSet_2(\tcA),\ \ \ \tA(z)=z+\frac1{2}\big(\tA(z)^2+\tA(z^2)\big)
\]
To obtain a leap generator for $\tcA_n$ one can proceed similarly to P\'olya trees. 
We consider the labeled class $\cC=\Sym(\cA)$, related to $\tcA$ by $\cC_n\simeq n!\times \tcA_n$. For $(\sigma,\gamma)\in\cC_n$ we note that $\sigma$ naturally extends to a permutation of the nodes, so that $\sigma$ can be seen as a permutation on the vertex-set of $\gamma$. As before, the \emph{core} of $\gamma$ is the subtree~$\tau$ formed by the vertices that are fixed by $\sigma$. Letting $\cB$ be the subclass of $\cC$ such that the root-vertex is the only fixed vertex, we note that each leaf $v$ of $\tau$ has a (possibly trivial) hanging subtree in $\cB$. So the core-decomposition yields
\[
\cC=\cA\circ\cB.
\]
Again this is reflected at the generating function level: the series $\tA(z)$ is the exponential generating function of $\cC$, and the above equation for $\tA(z)$ rewrites as $\tA(z)=A(B(z))$, where $A(z)=z+\frac1{2}A(z)^2$ and  $B(z):=z+\frac1{2}\tA(z^2)$ is the exponential generating function of $\cB$. 

This composition scheme is known to be admissible, with $\rho_C\approx 0.403$ and ${\rho_B=\rho_C^{1/2}}$. The Boltzmann sampler for $\cB$ is deduced from the known Boltzmann sampler $\Gamma \tilde{A}(x)$ for $\tcA$~\cite{flajolet2007boltzmann}. It is shown in Figure~\ref{algo:gamma-B2}, and its complexity is linear in the size of the output. Exact-size sampling for the class $\cA$ is also easy to obtain (e.g. by Pr\"ufer encoding, or R\'emy's algorithm~\cite{remy1985procede}). Thus the leap generator for $\cC$ has linear time complexity.

\begin{figure}[t]\small
\fbox{\begin{minipage}{.82\textwidth}

\begin{tabbing}
XXX \= XX \= XXX \= XXX \= XXX \kill
{\bf Algorithm } $\Gamma B(x)$~:\\
\>{\bf if} $\mathrm{Bern}(x/B(x))$ {\bf then~return} $\mathcal Z$\\
\>{\bf else}\\
\>\> $\gamma \leftarrow \Gamma \tilde{A}(x^2)$\\
\>\> {\bf return} $[\mathcal Z, [\gamma,\gamma]]$
\end{tabbing}
\end{minipage}}

\caption{\label{algo:gamma-B2}
Algorithm $\Gamma B(x)$ for phylogenetic trees}
\end{figure}

\subsubsection{Unlabeled $k$-ary mobiles}
A \emph{mobile} is a rooted tree where the children of each node are cyclically ordered. 
For fixed $k\geq 2$, let $\cAk$ be the class of $k$-ary mobiles with labeled leaves (the size being the number of leaves), which is given by
\[
\cAk=\cZ+\Cyc_k(\cA),\ \ \ \Ak(z)=z+\frac1{k}\Ak(z)^k,
\]
with the following explicit expression for the counting coefficients:
\[
|\cAk_n|=\frac{(m+n-1)!}{m!k^m}\ \mathrm{for}\ m:=\frac{n-1}{k-1}\in\mathbb{N},\ \ \ |\cAk_n|=0\ \mathrm{for}\ n\neq 1\ \mathrm{mod}\ k-1. 
\]
Note that this is an extension of the previous section, since the case $k=2$ gives  phylogenetic trees (indeed, for $k=2$, $\Cyc_k$ and $\Set_k$ are the same construction). 

Let $\tcA$ be the associated unlabeled class, given by (where $\Cyc_k$ is the construction for unlabeled classes)
\[
\tcA=\cZ+\Cyc_k(\tcA),\ \ \ \tA(z)=z+\frac1{k}\sum_{d|k}\phi(d)\tA(z^d),
\] 
with $\phi(.)$ the Euler totient function. 
As before we consider the labeled class $\cC=\Sym(\cA)$, and for $(\sigma,\gamma)\in\cC$, extend $\sigma$ to the whole vertex-set of $\gamma$, and define the core as the subtree made of the vertices (nodes or leaves) that are fixed under $\sigma$. Letting $\cB$ be the subclass of $\cC$ such that the root is the only fixed vertex, we have 
\[
\cC=\cA\circ\cB.
\]
again reflected at the generating function level by the fact that $C(z):=\tA(z)$ is the exponential generating function of $\cC$, 
and $B(z):=z+\frac1{k}\sum_{d>1,d|k}\phi(d)\tA(z^d)$ is the exponential generating function of $\cB$. 

\begin{remark}
For $k\geq 3$ we have not succeeded in applying a similar strategy to the class $\cAk$ of rooted non-embedded $k$-ary trees, which is given by
\[
\cAk=\cZ+\Set_k(\cAk),\ \ \ \Ak(z)=z+\frac1{k!}\Ak(z)^k,
\]
The unlabeled class $\tcAk$ is specified by
\[
\tcAk=\cZ+\MSet_k(\tcAk),\ \ \ \tAk(z)=\Pk(\tAk(z),\ldots,\tAk(z^k)),
\]
where
\[
\Pk(s_1,\ldots,s_k)=[u^k]\exp\Big(\sum_{i\geq 1}\tfrac1{i}u^is_i\Big).
\]
For instance for $k=3$, letting $C(z)=A^{(3)}(z)$, we have 
\[
C(z)=z+\frac1{6}C(z)^3+\frac1{2}C(z)C(z^2)+\frac1{3}C(z^3).
\] 
The term involving $C(z)C(z^2)$
makes it difficult to write $C(z)$ as $A(B(z))$ with $\rho_B>\rho_C$ and the coefficients of $A(z)$ having a simple closed form (or at least such that the functional inverse of $A(z)$ is rational). 
\end{remark}

\subsubsection{Unlabeled Schr\"oder mobiles}
A \emph{Schr\"oder mobile} is a mobile with no node having one child; the size is the number of leaves. The class $\cA$ of Schr\"oder mobiles that are labeled at leaves is given by
\[
\cA=\cZ+\Cyc_{\geq 2}(\cA), \ \ A(z)=z+\log\Big(\frac1{1-A(z)}\Big)-A(z).
\] 
There is no simple closed formula for the counting coefficients, but the equation for $A(z)$ can be 
written as $A(z)=z\Phi(A(z))$, with $\Phi(y)=\frac{1}{1-\tfrac1{y}\big(\log(1/(1-y))-y\big)}$. Thus the class $\cA$ fits into the setting of Galton-Watson trees for random generation. Devroye's algorithm~\cite{devroye2012simulating} allows to sample uniformly at random in $\cA_n$ in linear time. 

The associated unlabeled class $\tcA$ is given by (with $\Cyc_{\geq 2}$ the construction for unlabeled classes) 
\[
\tcA=\cZ+\Cyc_{\geq 2}(\tcA), \ \ \tA(z)=z+\sum_{d\geq 1}\frac{\phi(d)}{d}\log\Big(\frac1{1-\tA(z^d)}\Big)-\tA(z).
\]
Again one can consider the symmetry class $\cC=\Sym(\cA)$, the core of an object $(\sigma,\gamma)\in\cC$; and letting $\cB$ be the subclass of $\cC$ such that the root is the only fixed vertex, one has $\cC=\cA\circ\cB$. The equation for $\tA(z)$ reflects this decomposition: $C(z):=\tA(z)$ is the exponential generating function of $\cC$, and $B(z):=z+\sum_{d\geq 2}\frac{\phi(d)}{d}\log\Big(\frac1{1-\tA(z^d)}\Big)$ is the exponential generating function of $\cB$.

\subsection{Planar maps}\label{sec:maps}
A planar map, hereafter simply called a map, is a connected graph embedded on the sphere, considered up to continuous deformation. It is rooted by marking and orienting an edge that is called the root. 
In this section we show that leap generators can be successfully applied to known composition schemes relating map families.

\subsubsection{Planar maps decomposed into blocks}\label{sec:maps_blocks}
We demonstrate the strategy for maps decomposed into blocks, before discussing extensions in Section~\ref{sec:maps_extensions}. A map is called 2-connected if its edge-set $E$ cannot be partitioned as $E=E_1\cup E_2$ such that there is a unique vertex (called a separating vertex) incident to at least one edge from both subsets. In a map, a \emph{block} is a maximal 2-connected submap. Let~$\cC$ and $\cA$ be the families of rooted maps and of rooted 2-connected maps. 

\begin{figure}
\begin{center}
\includegraphics[width=12cm]{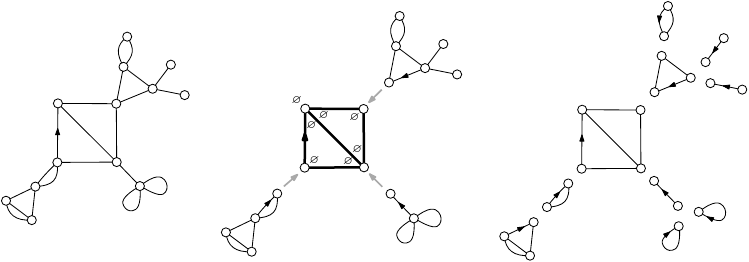}
\end{center}
\caption{Left: a rooted map $\gamma$. Middle: $\gamma$ is decomposed into a core (block containing the root-edge) where at each corner is attached a (possibly empty) rooted map. Right: the collection of blocks of $\gamma$. The block sizes in decreasing order are $5,4,3,2,2,1,1,1,1,1.$}
\label{fig:blocks}
\end{figure}

For a rooted map, the \emph{root-block} is the block containing the root-edge. It is well known~\cite{tutte1963census} that a rooted map can be decomposed into its root-block and a collection of (possibly empty) rooted maps, one attached at each corner of the root-block, see Figure~\ref{fig:blocks}. 
Since a rooted map with $n$ edges has $2n$ corners, this decomposition yields
\[
\cC=\cA\circ\cB,
\] 
where $\cB=\cZ\times(1+\cC)^2$.

The generating functions $C(z), B(z)$ and $A(y)$ have explicit algebraic expressions, giving the following singular expansions~\cite[Prop.~4]{banderier2001random} of order $3/2$, at $\rho=\rho_C=\rho_B=1/12$ and $\rho_A=4/27$ respectively,
\[
C(z)= \frac1{3} -\frac{4}{3}(1-12z)+\frac{8}{3}(1-12z)^{3/2}+O\big((1-12z)^{2}\big),
\] 
\[
B(z)= \frac{4}{27} -\frac{4}{9}(1-12z)+\frac{16}{27}(1-12z)^{3/2}+O\big((1-12z)^{2}\big),
\] 
\[
A(y)= \frac1{3} -\frac{4}{9}(1-27y/4)+\frac{8\sqrt{3}}{81}(1-27y/4)^{3/2}+O\big((1-27y/4)^{2}\big).
\]
Note that the composition scheme here is critical, namely the singularities of $C(z)$ and $B(z)$ are both at $\rho=1/12$ (whereas in the supercritical case we had $\rho_B>\rho_C$), and $B(\rho)=R$. From the singular expansion of $B(z)$ we derive that the expected size $\mu$ of an object in $\cB$ under the critical Boltzmann distribution is given by
\[
\mu=\frac{\rho B'(\rho)}{B(\rho)}=3.
\]

With the notation of Section~\ref{sec:singular}, the singular expansions of $A(y)$ and $C(z)$ give\footnote{Note that $\kappa_C/\kappa_A=\mu^{5/2}$, even if the singular exponent is $s=-3/2$. This is thus different from supercritical schemes where we had $\kappa_C/\kappa_A=\mu^{-s}$. The relation $\kappa_C/\kappa_A=\mu^{5/2}$ can be considered as  universal for map composition schemes, it is related to the property that there is almost surely a giant component (here a giant block), as revealed by the calculation in~\cite[Appendix~D]{banderier2001random}.} $\kappa_C=\frac{8}{3}$ and $\kappa_A=\frac{8\sqrt{3}}{81}$.

As before, in order to get a leap generator, we first need an exact-size linear-time random generator $\Gamma A[k]$ for the core-class of 2-connected maps. Such a generator can be derived from a bijection with ternary trees~\cite[Sec.~2.3.3]{Schaeffer:these}. We also need  a Boltzmann sampler $\Gamma B(z)$ for $\cB$ (to be called at the singularity $\rho$), such that the expected cost of generating an object $\gamma\in\cB$ is linear in the size of $\gamma$. Such a sampler can be assembled for a map family composed from a core-family having an exact-size linear-time sampler, see e.g.~\cite{Fusy09}. Here we can thus obtain such a Boltzmann sampler for $\cC$ (rooted maps) and for $\cB$ (pairs of rooted maps).  

Then the leap generator $\Gammal C[n]$ proceeds exactly in the same way as in the supercritical case, up to a slight adjustment, namely at the end we reroot the generated map  at a randomly chosen edge (which is given a random direction). By the same arguments as in the supercritical case, the expected complexity of each trial is $\Theta(n)$, and the probability of success of each trial is $\Theta(1)$ (see Lemma~\ref{lem:success_map} below), hence the expected complexity of $\Gammal C[n]$ is $\Theta(n)$.  
As before the distribution on $\cC_n$ given by the leap generator is called the \emph{leap distribution} and is denoted by $\pi_n'$.

\begin{lem}\label{lem:success_map}
Let $q_n$ be the probability of success at each trial of the leap generator. 
Then
\[
q_n=[z^n]\frac{1}{1-B(\rho z)/B(\rho)}=\frac1{\mu}+O(1/\sqrt{n}).
\]
\end{lem}
\begin{proof}
The proof that $q_n=[z^n]\frac{1}{1-B(\rho z)/B(\rho)}$ is the same as in the supercritical case. From the singular expansion of $B(z)$ we then obtain 
\[
\frac{1}{1-B(\rho z)/B(\rho)}=\frac{B(\rho)}{\rho B'(\rho)}\frac1{1-z/\rho}+O\Big(1/\sqrt{1-z/\rho}\Big).
\] 
By transfer theorems this yields the estimate of $q_n$. 
\end{proof}

Let $a_k=[y^k]A(y)$ and $c_n=[z^n]C(z)$. 
Let $\Pi_n$ be the set of  sequences $(k_1,k_2,\ldots,k_r)$ of positive integers adding up to $n$ and in (weakly) decreasing order. For $\kk=(k_1,\ldots,k_r)\in\Pi_n$ we let $\cC_{n,\kk}$ be the set of rooted maps in $\cC_n$ with $r$ blocks whose sizes in decreasing order are given by $\kk$, and let $c_{n,\kk}=|\cC_{n,\kk}|$.   

\begin{lem}
Let $\kk=(k_1,\ldots,k_r)\in\Pi_n$.   
Then for every $\gamma\in\cC_{n,\kk}$ one has
\begin{equation}\label{eq:pi}
\pi_n'(\gamma)=\frac{1}{|\cC_n|}d_{n,\kk},\ \ \mathrm{with}\ \ d_{n,\kk}:=\frac1{q_n}\sum_{i=1}^r\frac{c_n\rho^n}{a_{k_i}B(\rho)^{k_i}}\frac{k_i}{n}.
\end{equation} 

\end{lem}
\begin{proof}
As before, if we let $\pi_n^{\mathrm{first}}(\gamma)$ be the probability that $\gamma$ is drawn at the first trial, then we have 
$\pi_n'(\gamma)=\frac1{q_n}\pi_n^{\mathrm{first}}(\gamma)$.  
Let $b_1,\ldots,b_r$ be the blocks of $\gamma$ in size-decreasing order. 
For $i\in[1..r]$, there are $2k_i$ rerootings of $\gamma$ with the root in $b_i$. Let $M'$ be any of these rerooted maps (forgetting the primary root of $M$). By the same arguments as in the proof of Lemma~\ref{lem:distort}, the probability that $M'$ is drawn in the first trial equals  
$\frac{\rho^n}{a_{k_i}B(\rho)^{k_i}}$, and conditionally the probability that it is rerooted as $M$ is $1/(2n)$. Hence the contribution to $\pi_n'(\gamma)$ from first rooting in $b_i$ is equal to  $\frac{2k_i}{2n}\frac{\rho^n}{a_{k_i}B(\rho)^{k_i}}$.
\end{proof}

We then have
\begin{equation}\label{eq:dtv_maps}
\dtv(\pi_n,\pi_n')=\frac1{2}\sum_{\kk\in\Pi_n}\sum_{\gamma\in\cC_{n,\kk}}\Big|\frac1{c_n}-\pi_n'(\gamma)\Big|=\frac1{2}\sum_{\kk\in\Pi_n}\frac{c_{n,\kk}}{c_n}|1-d_{n,\kk}|.
\end{equation}

Our aim is to show that $\dtv(\pi_n,\pi_n')\sim d/n^{1/3}$ for some explicit constant $d>0$.     
We first need some estimates on the distortion factor $d_{n,\kk}$. 

\begin{lem}
Let $\od_{n,k}:=\frac1{q_n}\frac{c_n\rho^n}{a_kB(\rho)^k}\frac{k}{n}$. Then, for $k=n/\mu+xn^{2/3}$,  
\begin{equation}\label{eq:asympt_dnk_maps}
\od_{n,k}= 1 + \frac{7}{2}\mu x n^{-1/3}+O(n^{-1/2}). 
\end{equation}
uniformly for $x$ in any fixed compact. 
Moreover, for $k=\Theta(n)$,  
\begin{equation}\label{eq:estimate_dnk_maps}
\od_{n,k}=1+O(xn^{-1/3})+O(n^{-1/2}).
\end{equation}


For $\kk\in\Pi_n$, let $e_{n,\kk}=d_{n,\kk}-\od_{n,k_1}=\frac1{q_n}\sum_{i=2}^r\frac{c_n\rho^n}{a_{k_i}B(\rho)^{k_i}}\frac{k_i}{n}$. 
Then there is a universal constant $K>0$ such that (with the convention $k_2=0$ if $r=1$)
\begin{equation}\label{eq:boundenk}
e_{n,\kk}\leq K (k_2/n)^{5/2}.
\end{equation}
\end{lem}
\begin{proof}
We have (using the fact that $\kappa_C/\kappa_A=\mu^{5/2}$)
\begin{align}
\od_{n,k}&=\frac1{q_n}\frac{c_n\rho^n}{a_kB(\rho)^k}\frac{k}{n}\\
&=\big(\mu+O(n^{-1/2})\big)\cdot \frac{\kappa_C}{\kappa_A}(k/n)^{7/2}\big(1+O(1/n)\big)\\
&=\mu\frac{\kappa_C}{\kappa_A}\frac1{\mu^{-7/2}}(1+x\mu n^{-1/3})^{7/2}\big(1+O(n^{-1/2})\big)\\
&=(1+x\mu n^{-1/3})^{7/2}\big(1+O(n^{-1/2})\big).
\end{align}
which yields the first two statements. 

Regarding the third statement, since $a_kB(\rho)^k=\Theta(k^{-5/2})$ and $c_n\rho_n=\Theta(n^{-5/2})$, there exists a universal constant $K'$ such that
\[ 
\Lambda:=\sum_{i=2}^r\frac{c_n\rho^n}{a_{k_i}B(\rho)^{k_i}}\frac{k_i}{n}\leq K'n^{-7/2}\sum_{i=2}^rk_i^{7/2}. 
\]
For $\beta>1$ let $s_{n,m}(\beta)$ be the maximum of $S=\sum_ik_i^{\beta}$ over all tuples of positive integers in decreasing order, whose first summand is $m$,  
and whose sum is bounded by $n$. 
Then clearly $\Lambda\leq K'n^{-7/2}s_{n,k_2}(7/2)$. 
It is then easy to check by induction on the number of summands that
 $s_{n,m}(\beta)\leq n\ \!\,m^{\beta-1}$. Indeed, the bound is clearly true for tuples with one summand. Then, for at least two summands, letting $S'$ be the sum of the $k_i^{\beta}$ without the first summand, 
  by induction we have 
 \[
 S=m^{\beta}+S'\leq m^{\beta}+ (n-m) m^{\beta-1}\leq n\ \,m^{\beta-1}.
 \]
 This yields $\Lambda\leq K'n^{-7/2}s_{n,k_2}(7/2)\leq K'(k_2/n)^{5/2}$, 
and concludes the proof of the third statement, using the fact that $1/q_n=O(1)$. 
\end{proof}

Hereafter we denote by $X_n$ the size of the largest block, and $Y_n$ the size of the second largest block (with the convention $Y_n=0$ if there is a single block) in a random map of size $n$. The following limit statement is proved in \cite[Theo.~7]{banderier2001random}, where $\Ai(x)=\frac1{\pi}\sum_{k\geq 1}(-1)^{k-1}x^k\frac{\Gamma(1+2k/3)}{\Gamma(1+k)}\mathrm{sin}(2\pi k/3)$ is the so-called map-Airy probability distribution on $\mathbb{R}$.

\begin{lem}\label{lem:asympt_Pnk}
 Let $\mu=3$ and $c=\frac{3}{4}2^{2/3}$. Then for $k=n\mu+xn^{2/3}$ with $x$ in a fixed compact, we have
\[
\mathbb{P}(X_n=k)\sim n^{-2/3}c\ \!\Ai(cx).
\]
\end{lem}

We also need tail-bound estimates on $X_n$ and $Y_n$, as stated in the next two lemmas. 
A sequence $u_n\geq 0$ is called \emph{stretched-exponential} if $u_n=O(\exp(-n^a))$ for some positive constant $a$. 

\begin{lem}\label{lem:addblocks}
Let $\epsilon>0$ and $m\geq 1$ be fixed. Let $\cE_{n}$ be the event that, in a  random rooted map with $n$ edges, there exist $m$ blocks $b_1,\ldots,b_m$ such that $|b_1|+\cdots+|b_m|>n/\mu+n^{2/3+\epsilon}$.   
Then $\mathbb{P}(\cE_{n})$ is stretched-exponential. 
\end{lem} 
\begin{proof}
This is easy to establish via the probabilistic approach of~\cite{AB19}. To build a random  map of size $n$, one draws a sequence of $2n+1$ independent calls to critical Boltzmann samplers for the class $1+\cA$ of (possibly empty) 2-connected maps. If the sequence $s_1,\ldots,s_{2n+1}$  of doubled sizes of the drawn blocks is such that $s_1+\cdots+s_i\geq i$ for all $1\leq i\leq 2n$ and $s_1+\cdots+s_{2n+1}=2n$ (Łukasiewicz condition) then one returns the map made of the blocks assembled along the tree that is associated to the Łukasiewicz walk. Moreover, the Łukasiewicz condition is met with probability $\Theta(n^{-5/2})$. Let $Z$ be a random variable giving the doubled size of a 2-connected map under the Boltzmann distribution at $\rho_A=4/27$; and denote by $Z_n$ the sum of $n$ independent copies of $Z$.  
Then, by the above discussion and a union-bound,
\[
\mathbb{P}(\cE_{n})\leq\Theta(n^{5/2})\binom{2n+1}{m}\mathbb{P}\big(Z_{2n+1-m}<2(n-n/\mu-n^{2/3+\epsilon})\big),
\]
The variable $Z\geq 0$ satisfies $\mathbb{E}(Z)=\frac{2\rho_{A}A'(\rho_A)}{1+A(\rho_A)}=2/3$ and the asymptotics $a_k\rho_A^k=\Theta(k^{-5/2})$ ensures that $\mathbb{P}(Z>i)=\Theta(i^{-3/2})$. The event $\tilde{\cE}_n:=\{Z_{2n+1-m}<2(n-n/\mu-n^{2/3+\epsilon})\}$ is thus equivalent
 to $\{Z_{2n+1-m}<(2n+1-m)\mathbb{E}(Z)+\eta-2n^{2/3+\epsilon}\}$, with $\eta=\frac{2m-2}{3}$ fixed.  
 Large deviation results (see e.g.~\cite[Theo.1(iii)b]{gao1999size}) then ensure that $\mathbb{P}\big(\tilde{\cE}_n\big)$ is stretched-exponential. Hence $\mathbb{P}(\cE_{n})$ is also stretched-exponential. 
\end{proof}

\begin{lem}\label{lem:estimate_Pnk}
Let $E_n=[n/10,n/2]$. Then $\mathbb{P}(X_n\notin E_n)=O(1/n^{1/2})$. 
Furthermore, for $k\in E_n$, letting $x=(k-n/\mu)/n^{2/3}$, we have
\begin{equation}\label{eq:estime_Pnk}
\mathbb{P}(X_n=k)=\frac{1}{n^{2/3}}O\big((1+|x|)^{-5/2}\big).
\end{equation}
\end{lem}
\begin{proof}
It is shown in the proof of Lemma~3.5 in~\cite{AB19} that $\mathbb{P}(X_n\leq n/10)=O(1/n^{1/2})$, and it follows from Lemma~\ref{lem:addblocks} (for $m=1$) that $\mathbb{P}(X_n\geq n/2)$ is stretched-exponential. 
The estimate~\eqref{eq:estime_Pnk} follows\footnote{In fact it gives~\eqref{eq:estime_Pnk} for $\mathbb{P}(\tilde{X}_n=k)$ where $\tilde{X}_n$ is the size of the 2-connected block containing the root-edge; since $\mathbb{P}(\tilde{X}_n=k)\geq \frac{k}{n}\mathbb{P}(X_n=k)$ one has $\mathbb{P}(X_n=k)\leq 10\,\mathbb{P}(\tilde{X}_n=k)$ for $k\in E_n$, hence~\eqref{eq:estime_Pnk} also holds for $\mathbb{P}(X_n=k)$.} from ~\cite[Theo.4]{banderier2001random} and asymptotics of the Airy function.   
\end{proof}

\begin{prop}\label{prop:dtv_maps}
We have
\[
\dtv(\pi_n,\pi_n')\sim \frac{d}{n^{1/3}}
\]
where $d=\frac{7}{2^{2/3}}\int_{\mathbb{R}} \Ai(x)|x|\mathrm{d}x\approx 3.03$. 
\end{prop}
\begin{proof}


Since $d_{n,\kk}=\od_{n,k_1}+e_{n,\kk}$, we have 
\[
\big|1-\od_{n,k_1}\big|-e_{n,\kk}\leq |1-d_{n,\kk}|\leq \big|1-\od_{n,k_1}\big|+e_{n,\kk}.
\] Hence,  
\[
D_1(n)-D_2(n)\leq \dtv(\pi_n,\pi_n')\leq D_1(n)+D_2(n),
\]
where 
\[
D_1(n):=\frac1{2}\sum_k\mathbb{P}(X_n=k)\ \!\big|1-\od_{n,k}\big|,\ \ D_2(n):=\frac1{2}\sum_{\kk\in\Pi_n}\frac{c_{n,\kk}}{c_n}e_{n,\kk}.
\]
We will show that $D_1(n)\sim d/n^{1/3}$, and then that $D_2(n)=o(n^{-1/3})$.

By the first statement in Lemma~\ref{lem:estimate_Pnk} we have $\sum_{k\notin E_n}\mathbb{P}(X_n=k)=O(n^{-1/2})$. Moreover  $\od_{n,k}=O(1)$ over all $n,k$. Hence $\frac1{2}\sum_{k\notin E_n}\mathbb{P}(X_n=k)\ \!\big|1-\od_{n,k}\big|=O(n^{-1/2})$. Hence, if we let $\tilde{D}_1(n):=\frac1{2}\sum_{k\in E_n}\mathbb{P}(X_n=k)\ \!\big|1-\od_{n,k}\big|$, then 
showing $D_1(n)\sim d/n^{1/3}$ reduces to showing $\tilde{D}_1(n)\sim d/n^{1/3}$. We fix $M>0$ and let $E_{n}(M)$ be the subset of $k\in E_n$ such that $x=(k-n/\mu)/n^{2/3}$ satisfies $|x|\leq M$. 
We let $\tilde{D}_1(n,M)$ be the contribution to $\tilde{D}_1(n)$ given by $k\in E_n(M)$. 
By Lemma~\ref{lem:asympt_Pnk} and~\eqref{eq:asympt_dnk_maps} we have 
\[
\tilde{D}_1(n,M)\sim d(M)/n^{1/3},
\]
where 
\[
d(M)=\frac{7\mu}{4}\int_{-M}^M c\ \!\Ai(cx)|x|\mathrm{d}x=\frac{7}{2^{2/3}}\int_{-cM}^{cM} \Ai(x)|x|\mathrm{d}x
\]
For $k\geq 0$, with $x=(k-n/\mu)/n^{2/3}$ the property $k\in E_n$ 
is equivalent to $x\in[-\frac{7}{30} n^{1/3},\frac1{6}n^{1/3}]$. Hence, by~\eqref{eq:estime_Pnk} and~\eqref{eq:estimate_dnk_maps}, letting $J_n=[-\frac{7}{30} n^{1/3},\frac1{6}n^{1/3}]$ and $J_{n,M}=J_n\backslash [-M,M]$, we have, for $n$ large enough,
\begin{align*}
\tilde{D}_1(n)-\tilde{D}_1(n,M)&=n^{-1/3}\cdot \ O\left(\int_{J_{n,M}} |x|^{-3/2}
\mathrm{d}x\right)\\
&=n^{-1/3}\ O\big( M^{-1/2}\big).
\end{align*}

We thus obtain
\[
\tilde{D}_1(n)=n^{-1/3}\left(d(M)+O(M^{-1/2})\right)
\]
Letting $M\to\infty$ we conclude that $\tilde{D}_1(n)\sim d/n^{1/3}$. 

Regarding $D_2(n)$, by~\eqref{eq:boundenk} we have
\[
D_2(n)\leq\frac{K}{2}\sum_{\kk\in\Pi_n}\frac{c_{n,\kk}}{c_n}(k_2/n)^{5/2}.
\]
and, using $\mathbb{P}(X_n\notin E_n)=O(n^{-1/2})$ we obtain
\[
D_2(n)\leq O(n^{-1/2})+\frac{K}{2}\sum_{\kk\in\Pi_n, k_1\in E_n}\frac{c_{n,\kk}}{c_n}(k_2/n)^{5/2}.
\]
By Lemma~\ref{lem:addblocks} (for $m=2$), there is a positive constant $a$ such that 
\[
\sum_{\kk\in\Pi_n, k_1\in E_n}\frac{c_{n,\kk}}{c_n}(k_2/n)^{5/2}\leq O(\exp(-n^a))+\tilde{D}_2(n)
\]
where
\[
\tilde{D}_2(n):=\sum_{k\in E_n}\mathbb{P}(X_n=k)n^{-5/2}\ \!\big|n/\mu-k+n^{2/3+\epsilon}\big|^{5/2}.
\]
Hence, showing $D_2(n)=o(n^{-1/3})$ reduces to showing $\tilde{D}_2(n)=o(n^{-1/3})$. 
Then~\eqref{eq:estime_Pnk} gives
\begin{align*}
\tilde{D}_2(n)&=n^{-5/2}O\left(\int_{J_n}\big|-xn^{2/3}+n^{2/3+\epsilon}\big|^{5/2}\cdot   
\Big((1+|x|)^{-5/2}\ \!\mathrm{d}x\right)\\
&=n^{-5/2}n^{5/3}O\left(   \int_0^{n^{1/3}} (x^{5/2}+n^{5\epsilon/2}) \cdot 
  (1+x)^{-5/2}\ \!\mathrm{d}x\right)\\
&= n^{-5/6}O\Big( n^{1/3} +n^{5\epsilon/2}\Big)\\
&=O(n^{-1/2}),\ \ \mathrm{taking}\ \epsilon\ \mathrm{small\ enough\ (e.g.\ }\epsilon=2/15\mathrm{)}.
\end{align*}

\end{proof}

\subsubsection{Extensions}\label{sec:maps_extensions}
Obtaining a leap generator for random maps is of minor interest, as there is a known exact-size bijective random generator for rooted maps of linear time complexity~\cite{schaeffer1997bijective}. For presentation purpose we chose to give the analysis of this case in detail. Leap generators can actually be applied to any critical map composition scheme $\cC=\cA\circ\cB$ where there is an exact-size linear time generator for the core-class $\cA$, and a Boltzmann sampler for $\cB$ of linear complexity in the size of the output (in general such a sampler for $\cB$ is derived from the composition scheme and the exact-size sampler for $\cA$). Letting 
\[
\mu=\frac{\rho B'(\rho)}{B(\rho)},\ \ \ c=\mu\Big( \frac{b_1}{3b_{3/2}}\Big)^{2/3},
\] 
where $B(z)=b_0-b_1(1-z/\rho)+b_{3/2}(1-z/\rho)^{3/2}+O\big((1-z/\rho)^2\big)$, 
the same analysis will give $\dtv(\pi_n,\pi_n')\sim d/n^{1/3}$, for
\[
d=\frac{7\mu}{4c}\int_{\mathbb{R}} \Ai(x)|x|\mathrm{d}x.
\]

One can thus obtain leap generators for composition schemes where no bijective linear-time  exact-size  sampler for $\cC$ is known. For instance, if we require 2-connected blocks to have at least two edges, then the composed class $\cC$ is made of maps that are bridgeless and loopless.   
As another extension, one can consider the model of block-weighted planar maps~\cite{bonzom2016large,stufler2020limits}, where each map is given a weight $u^{\mathrm{\# blocks}}$. It has been recently studied in detail~\cite{fleurat2024phase}, showing a phase transition at $u=5/9$: the composition scheme is critical for $u<5/9$ and subcritical for $u>5/9$. For $u<5/9$, we can thus obtain leap generators of linear time complexity.
Finally, the approach can also be applied to cubic planar graphs, where $\cA$ is the class of rooted 3-connected cubic planar graphs\footnote{Precisely, one restricts to the class $\cC$ of vertex-rooted cubic planar graphs whose root-vertex is on a 3-connected component. Upon unrooting, the uniform distribution on $\cC_n$ induces a distribution on cubic planar graphs of size $n$ that is asymptotically uniform, at total variation distance $O(n^{-1/2})$ from the uniform distribution.}.
 By Whitney's theorem and duality, this class is in bijection with rooted simple triangulations, for which a bijective random generator is known~\cite{Poulalhon:triang-3connexe+boundary}.

\section{Experiments}\label{sec:experiments}

We conducted several experiments illustrating the practical closeness of the leap distribution to the uniform distribution.

We start with some \emph{exact} experiments on Motzkin walks, for which total variation distance can be exactly computed up to large sizes
(in particular, this is possible thanks to the simple formula $c_{n,k}=\mathrm{Cat}_k\binom{n}{2k}$).  
\begin{figure}
  \centering
\includegraphics[width=12cm]{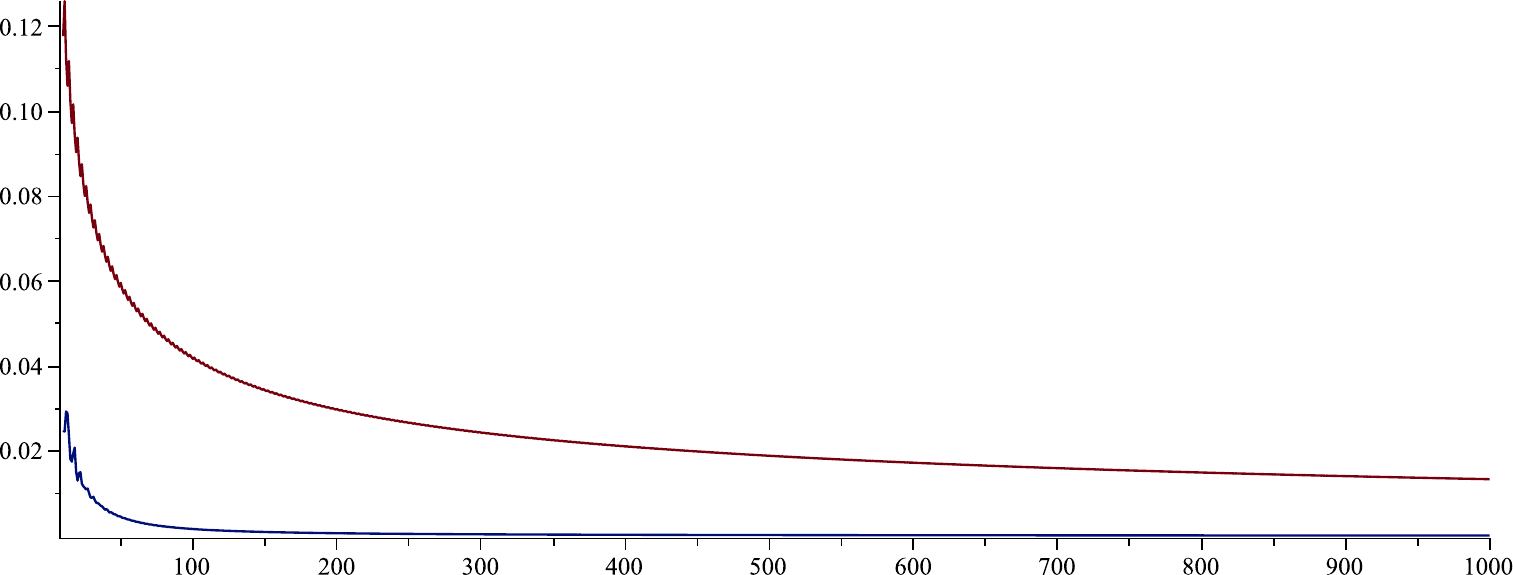}\\[.4cm]
  \caption{Plot of $d_n:=\dtv(\pi_n,\pi_n')$ (red) and $d_n^{\mathrm{rej}}:=\dtv(\pi_n,\pi_n^{\mathrm{rej}})$ (blue)  for Motzkin walks for $n$ from $10$ to $1000$.} \label{fig:dtvmotz}
\end{figure}
\begin{figure}
  \centering
\begin{subfigure}[]{\textwidth}  
\includegraphics[width=\textwidth]{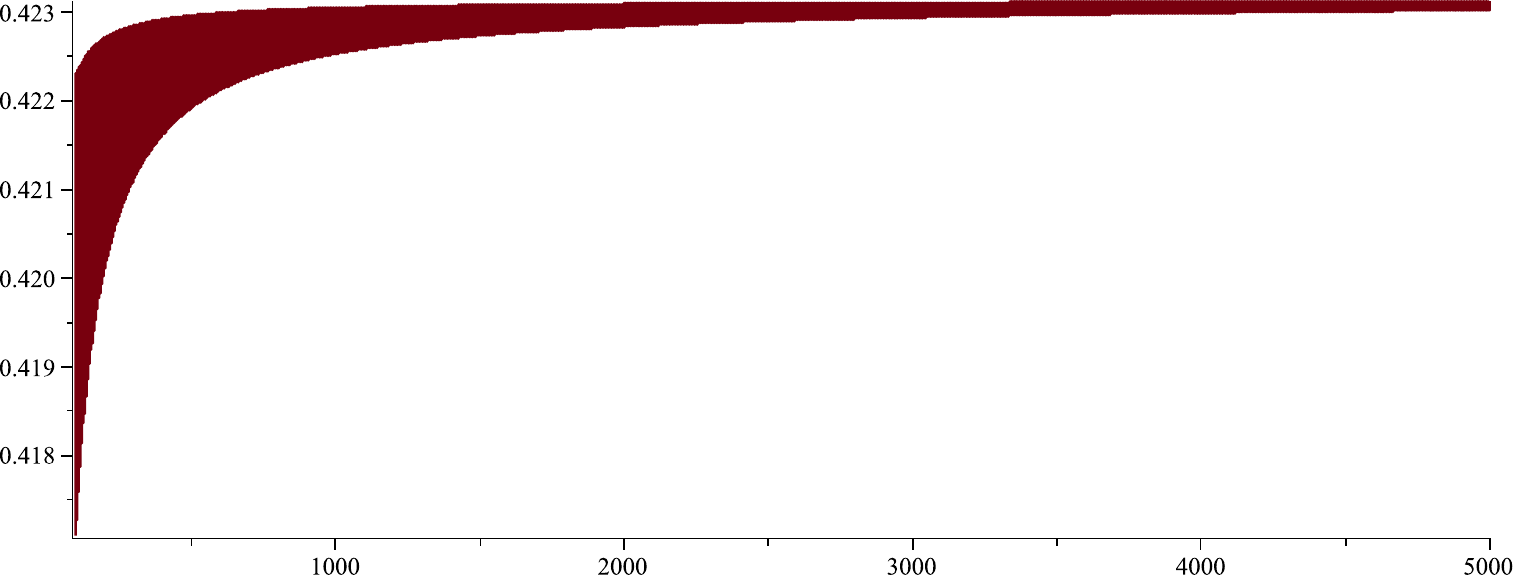}
\caption{}
\end{subfigure}  

\medskip
\medskip

\begin{subfigure}[]{\textwidth}  
\includegraphics[width=\textwidth]{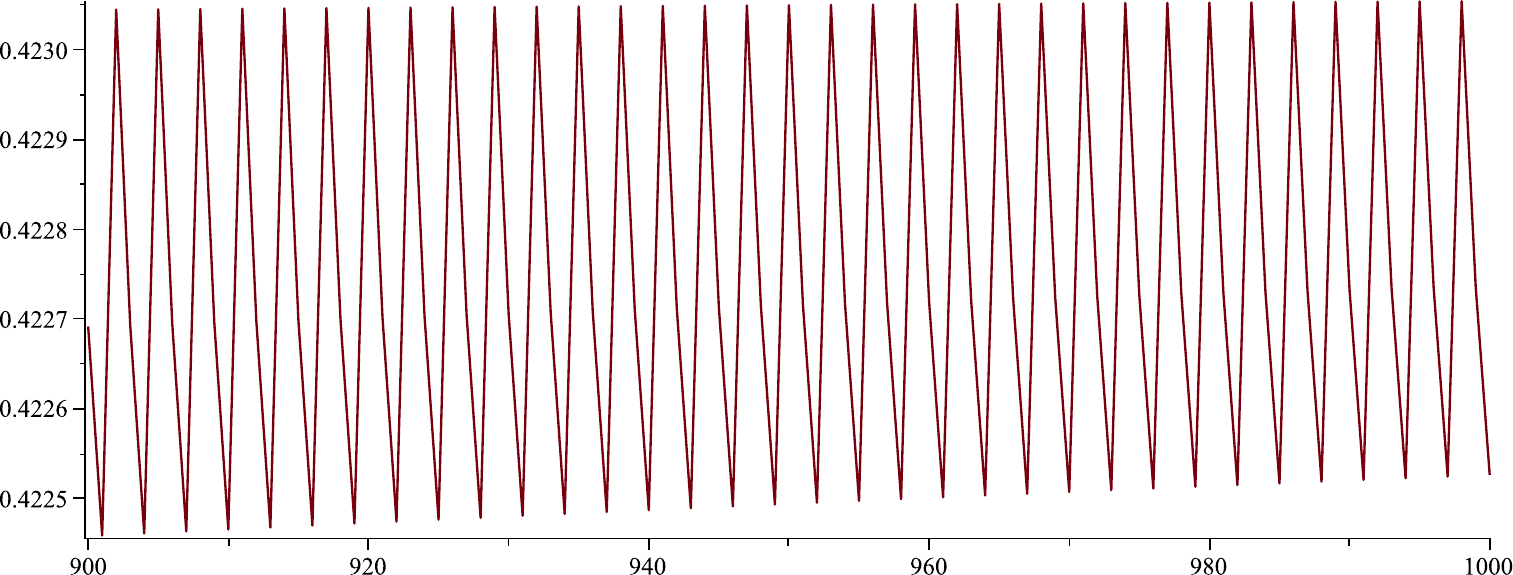} 
\caption{}
  \end{subfigure}  
  \caption{(a) Plot of $n^{1/2}d_n$ for Motzkin walks, for $n$ from $100$ to $5000$. (b)  The part of the plot for $n$ from $900$ to $1000$, showing fluctuations of period 3.}
  \label{fig:dtvmotzsqrt}
\end{figure}
\begin{figure}
  \centering
\includegraphics[width=12cm]{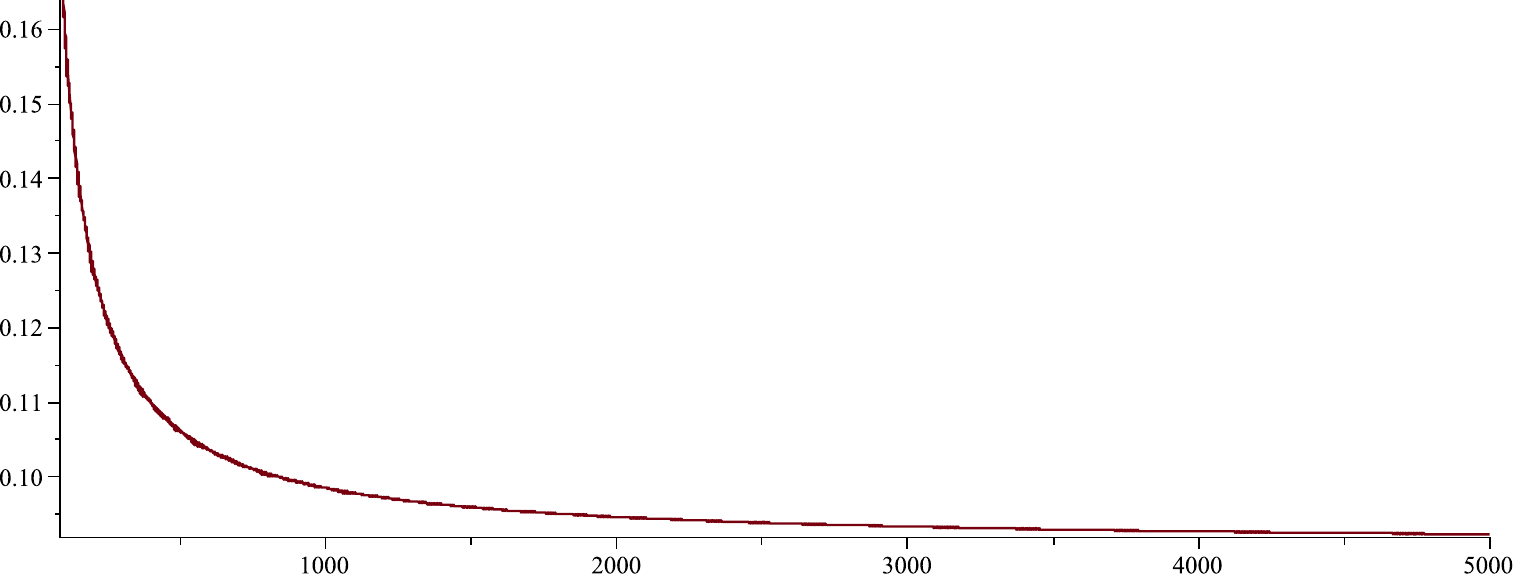}
  \caption{Plot of $n\ \!d_n^{\mathrm{rej}}$ for Motzkin walks, for $n$ from $100$ to $5000$. The horizontal axis is at the height $0.09074$ predicted by Proposition~\ref{prop:asympt_uniform_general} as the limit.} \label{fig:dtvaccmotzsqrt}
\end{figure}
Figure~\ref{fig:dtvmotz} shows the plot of $d_n:=\dtv(\pi_n,\pi_n')$ together with the plot of $d_n^{\mathrm{rej}}:=\dtv(\pi_n,\pi_n^{\mathrm{rej}})$ for first order acceleration of convergence. To verify Proposition~\ref{prop:asympt_uniform}, Figure~\ref{fig:dtvmotzsqrt} shows the plot of $n^{1/2}d_n$. We observe convergence to the  constant $\frac{3}{2}\frac{\sigma}{\sqrt{2\pi\mu}}\approx 0.42314$ predicted by Proposition~\ref{prop:asympt_uniform}, with some small fluctuations of period $3$. These are likely due to whether the theoretical mode $k_0=n/\mu$ (here $\mu=3$) of the core-size distribution occurs at an integer value. Similarly, to verify Proposition~\ref{prop:asympt_uniform_general}, Figure~\ref{fig:dtvaccmotzsqrt} shows the plot of $n\ \!d_n^{\mathrm{rej}}$. Again we see convergence to the constant $\approx 0.09074$ predicted by Proposition~\ref{prop:asympt_uniform_general}. 

By a slight abuse of notation, for $1\leq k\leq n$ we let $\pi_n(k)$ be the probability that the core-size is $k$ under the uniform distribution.  
Similarly we use the notation $\pi_n'(k)$ (resp. $\pi_n^{\mathrm{rej}}(k)$) for the analogous probability under the leap distribution (resp. under the leap distribution with first order acceleration of convergence). Note that $d_n=\frac1{2}\sum_k |\pi_n(k)-\pi'_n(k)|$ and $d_n^{\mathrm{rej}}=\frac1{2}\sum_k |\pi_n(k)-\pi^{\mathrm{rej}}_n(k)|$. 
For $n=5000$, 
Figure~\ref{fig:histogram_core_motz} shows in (a) the plots of $\pi_n(k),\pi_n'(k),\pi_n^{\mathrm{rej}}(k)$ superimposed with the 
Gaussian limit density function. Below, in (b) the plot of $d_{n,k}-1$ is approximately a line of positive slope, in agreement with Lemma~\ref{lem:estimate_dnk}. Therefore the leap distribution gives a slight bias toward a larger core-size, as also appears in the plot of $\pi'_n(k)-\pi_n(k)$ given in (c). On the other hand, in (d) the plot of $d_{n,k}^{\mathrm{rej}}-1$ is approximately a (convex) parabola with negative value at its center, in agreement with~\eqref{eq:dnkrej}. Accordingly, in the plot of $\pi^{\mathrm{rej}}_n(k)-\pi_n(k)$ shown in (e), there are 
two positive bumps on each side of the mode and a negative bump in between, and asymptotically the picture will tend to be mirror-symmetric at the mode.  Thus, on top of reducing the total variation distance, acceleration of convergence (at first order, and more generally at odd orders) makes the perturbation to the core-size (asymptotically) centered at the mode, thus better avoiding a drift effect.   

\begin{figure}
  \centering
\begin{subfigure}[]{\textwidth}
\includegraphics[width=\textwidth]{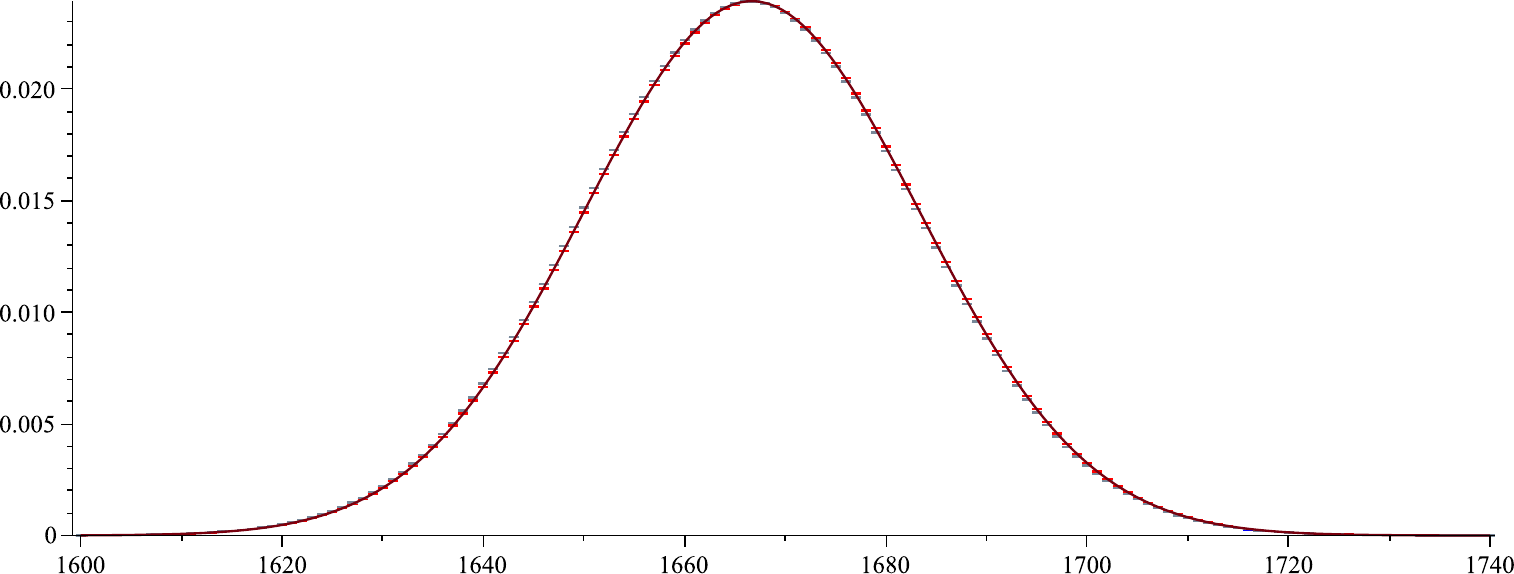}
\caption{}
\end{subfigure}

\bigskip
\medskip

\begin{subfigure}[B]{0.3\textwidth}
\includegraphics[width=\textwidth]{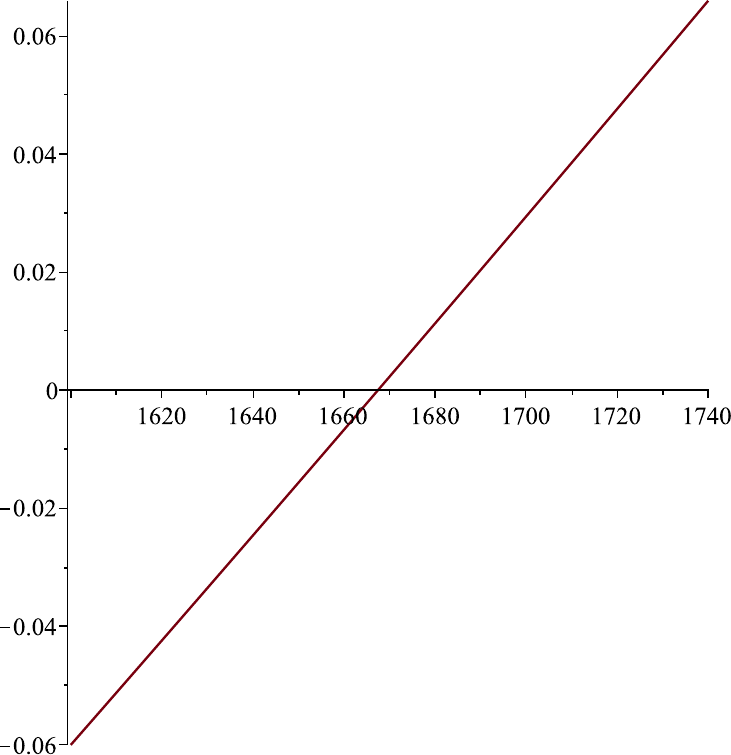}
\caption{}
\end{subfigure}
\hfill
\begin{subfigure}[B]{0.65\textwidth}
\includegraphics[width=\textwidth]{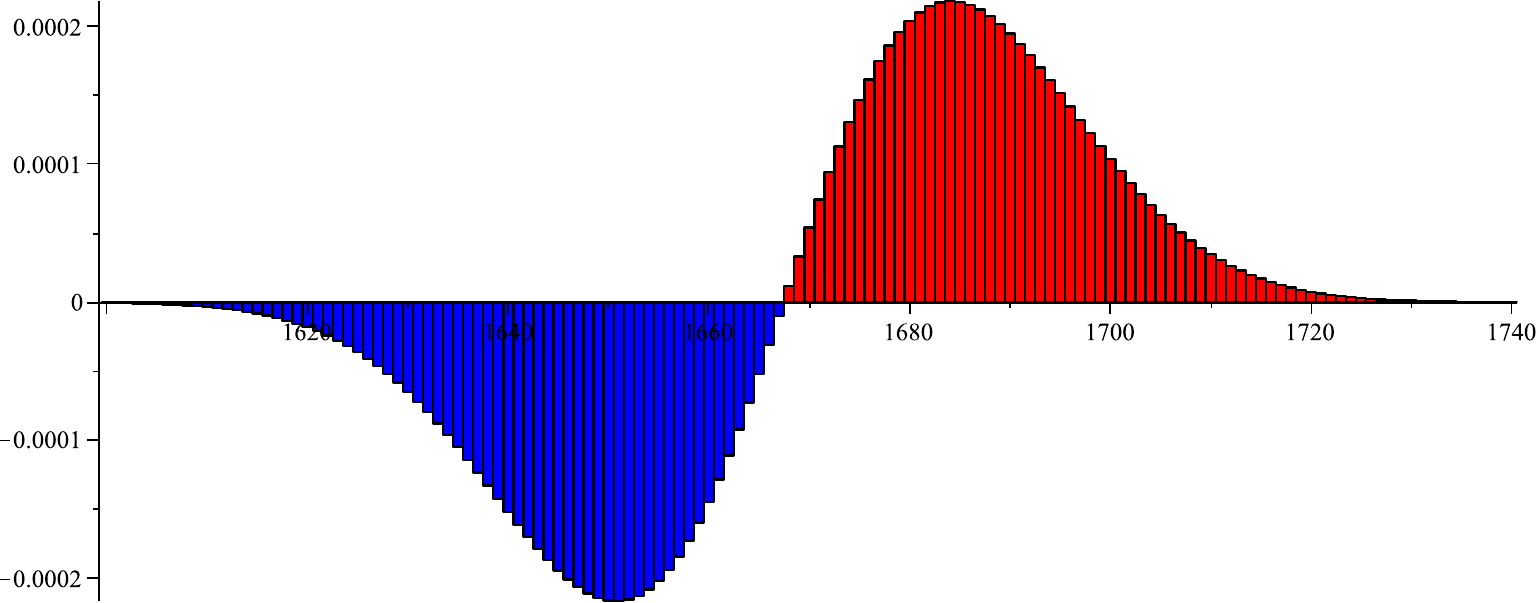}
\caption{}
\end{subfigure}

\bigskip
\medskip

\begin{subfigure}[B]{0.3\textwidth}
\includegraphics[width=\textwidth]{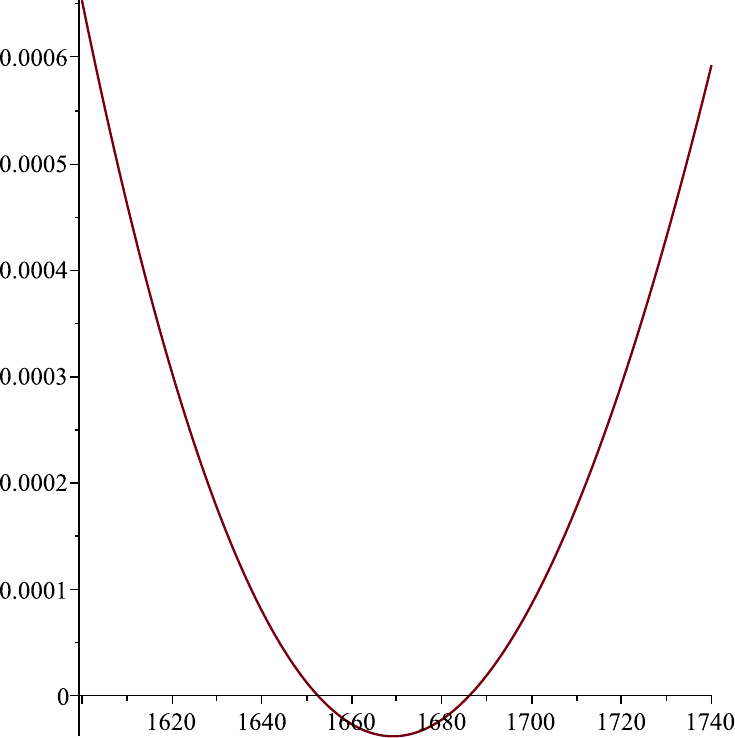}
\caption{}
\end{subfigure}
\hfill
\begin{subfigure}[B]{0.65\textwidth}
\includegraphics[width=\textwidth]{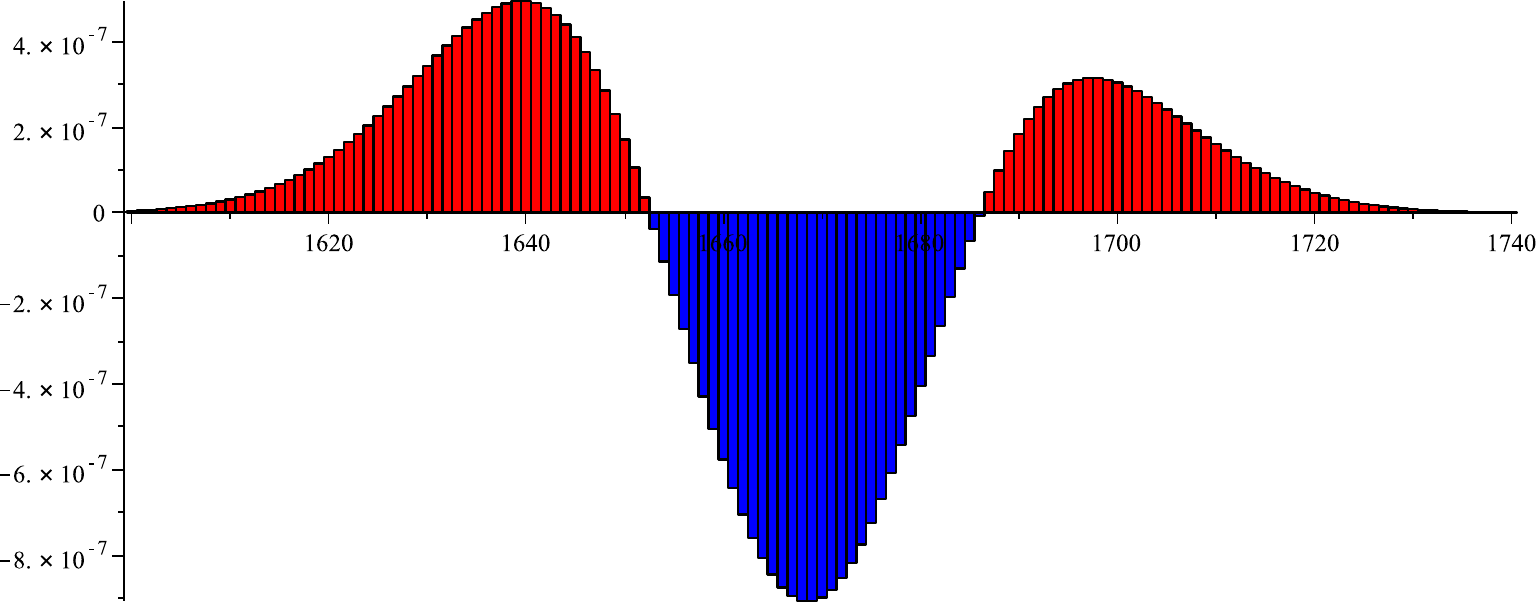}
\caption{}
\end{subfigure}
  \caption[]{Core-size of Motzkin walks of size $n=5000$: for $k$ in the vicinity of $n/\mu=5000/3$,  
  (a) shows the  plots of $\pi_n(k)$ (blue), $\pi_n'(k)$ (red) and $\pi_n^{\mathrm{rej}}(k)$ (gray, indistinguishable from $\pi_n(k)$), superimposed with the Gaussian limit density function, (b) the plot 
 of $d_{n,k}-1$, (c) the plot of $\pi_n'(k)-\pi_n(k)$, whose total absolute area equals $2\,d_n$, (d) the plot 
 of $d_{n,k}^{\mathrm{rej}}-1$, (e) the plot of $\pi_n^{\mathrm{rej}}(k)-\pi_n(k)$, whose total absolute area equals $2\, d_n^{\mathrm{rej}}$.}
  \label{fig:histogram_core_motz}
\end{figure}

\begin{figure}
  \centering
\includegraphics[width=12cm]{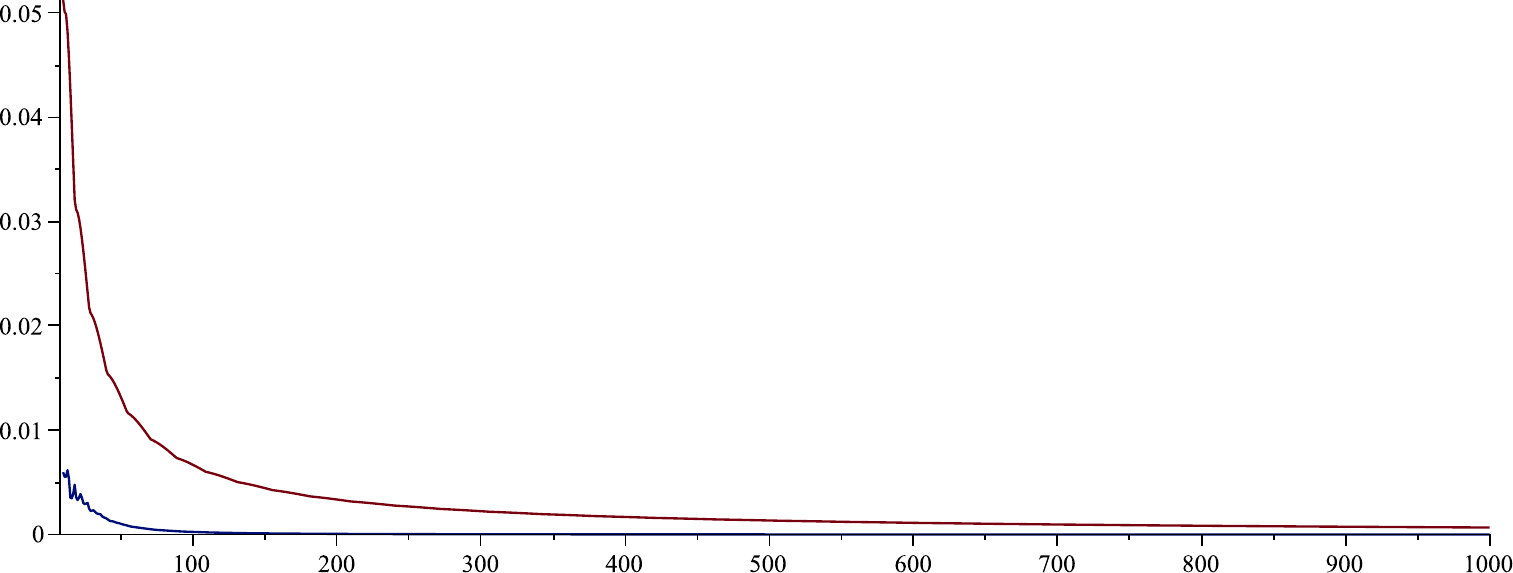}\\[.4cm]
  \caption{Total variation distance for the height of Motzkin walks: plots of $\hat{d}_n:=\dtv(\hat{\pi}_n,\hat{\pi}_n')$ (red) and $\hat{d}_n^{\mathrm{rej}}:=\dtv(\hat{\pi}_n,\hat{\pi}_n^{\mathrm{rej}})$ (blue), for $n$ from $10$ to $1000$. \bigskip} \label{fig:dtvmotzhauteur}
\end{figure}

\begin{figure}
  \centering
  \begin{subfigure}[]{\textwidth}
\includegraphics[width=\textwidth]{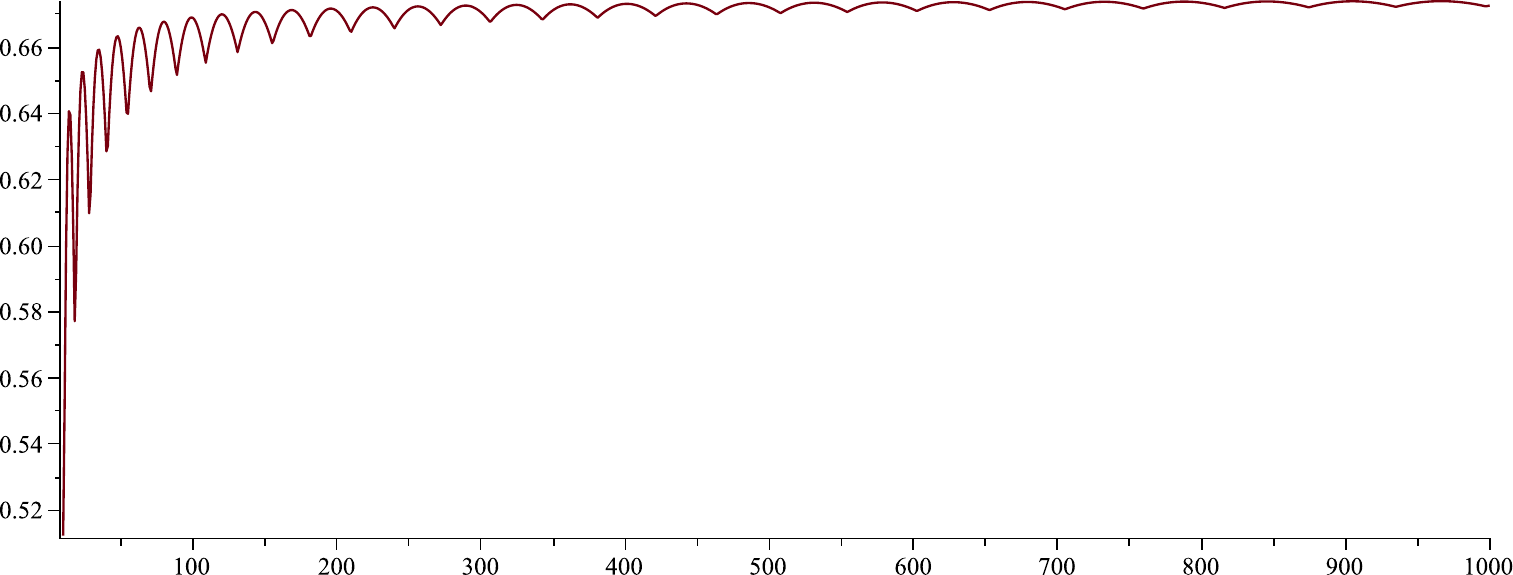}
\caption{}
\end{subfigure}

\medskip
\medskip

  \begin{subfigure}[]{\textwidth}
\includegraphics[width=\textwidth]{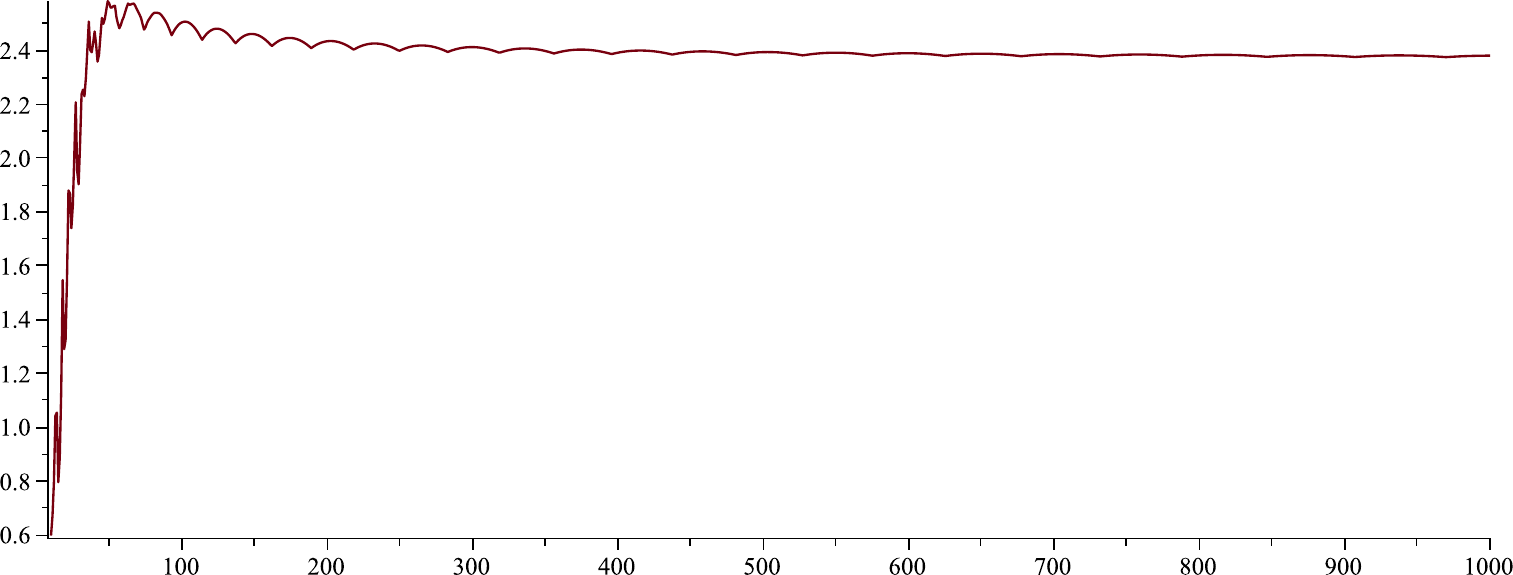}
\caption{}
\end{subfigure}
  \caption{Total variation distance for the height of Motzkin walks: plot of $n\ \!\hat{d}_n$ (a), and  $n^2\ \!\hat{d}_n^{\mathrm{rej}}$ (b), both for $n$ from $10$ to $1000$.} 
  \label{fig:dtvmotzheight}
\end{figure}

\begin{figure}
  \centering
\begin{subfigure}[]{\textwidth}
\begin{flushright}
\includegraphics[width=0.955\textwidth]{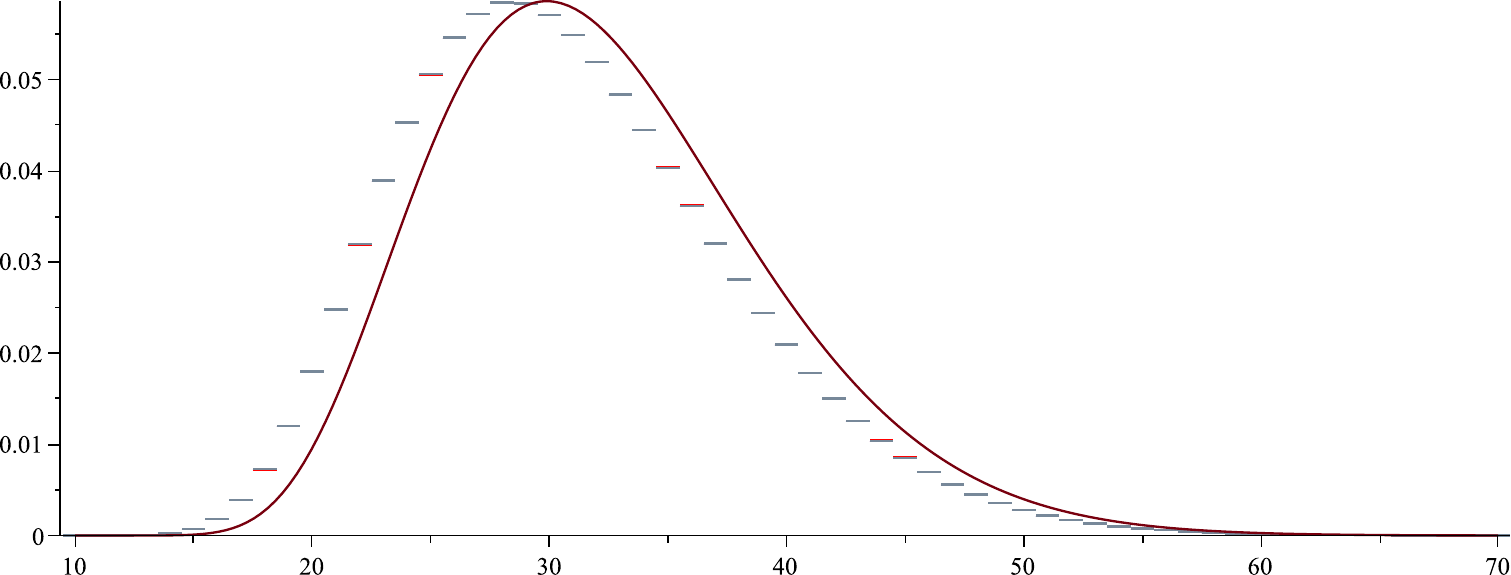}
\end{flushright}
\caption{}
\end{subfigure}

\medskip
\medskip

\begin{subfigure}[]{\textwidth}
\begin{flushright}
\includegraphics[width=0.99\textwidth]{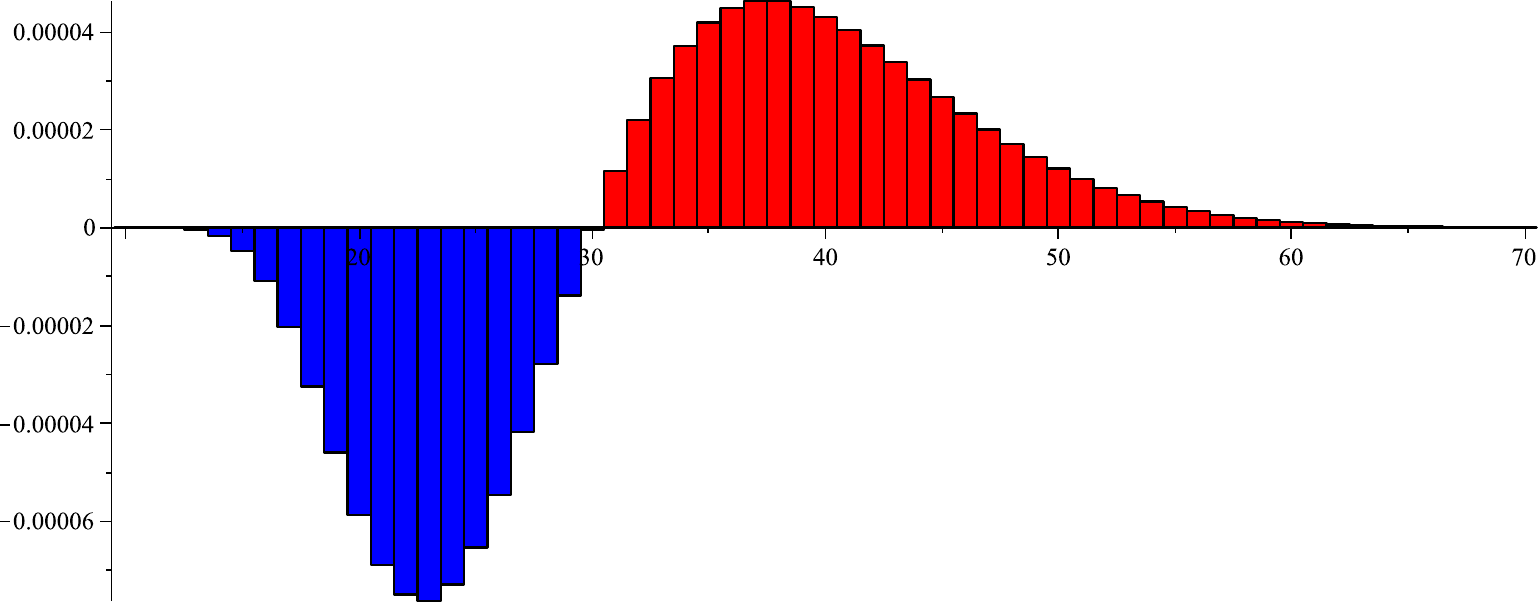}
\end{flushright}
\caption{}
\end{subfigure}

\medskip
\medskip

\begin{subfigure}[]{\textwidth}
\includegraphics[width=\textwidth]{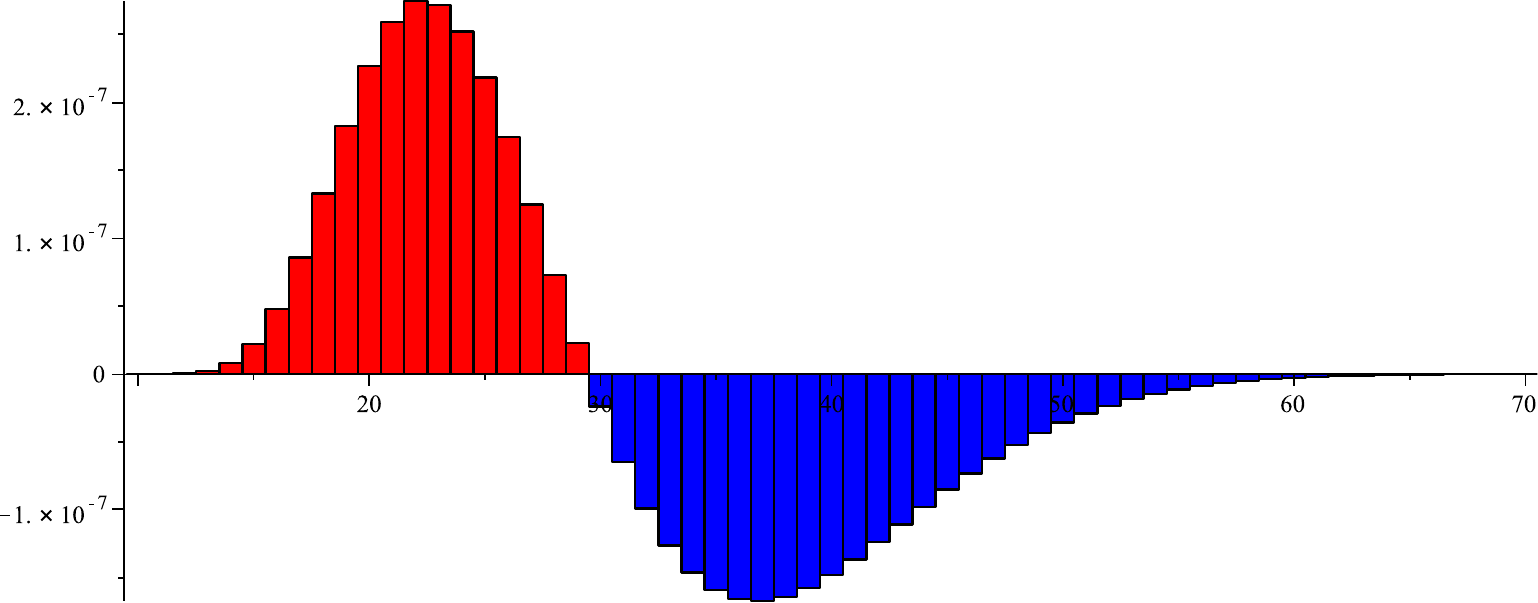}
\caption{}
\end{subfigure}
  \caption{
  Height in random Motzkin walks of size $n=1000$: (a) gives the (nearly indistinguishable) plots of $\hat{\pi}_n(h)$ (blue), $\hat{\pi}_n'(h)$ (red), and $\hat{\pi}_n^{\mathrm{rej}}(h)$ (gray), together with the limit curve; (b) the plot of $\hat{\pi}_n'(h)-\hat{\pi}_n(h)$, whose total area equals $2\,\hat{d}_n$; (c) the plot  
  of $\hat{\pi}_n^{\mathrm{rej}}(h)-\hat{\pi}_n(h)$, whose total area equals $2\,\hat{d}_n^{\mathrm{rej}}$.} 
  \label{fig:histogram_height_motz}
\end{figure}

On Motzkin walks we can also compute the total variation distance for the height parameter, using the fact that the height of a Motzkin walk equals the height of its Dyck core (for computations we need the coefficients $a_{k,h}$ giving the number of Dyck walks of length $2k$ and height $h$; these satisfy an explicit recurrence).  
Let $\hat{\pi}_n(h)$ (resp. $\hat{\pi}'_n(h)$, $\hat{\pi}^{\mathrm{rej}}_n(h)$) be the  
  distribution of the height under $\pi_n$ (resp. under $\pi_n'$, under $\pi_n^{\mathrm{rej}}$).   
Figure~\ref{fig:dtvmotzhauteur} shows the plots of $\hat{d}_n:=\dtv(\hat{\pi}_n,\hat{\pi}_n')=\frac1{2}\sum_h|\hat{\pi}_n(h)-\hat{\pi}_n'(h)|$ and of  
 $\hat{d}_n^{\mathrm{rej}}:=\dtv(\hat{\pi}_n,\hat{\pi}_n^{\mathrm{rej}})=\frac1{2}\sum_h|\hat{\pi}_n(h)-\hat{\pi}_n^{\mathrm{rej}}(h)|$ for $n$ from $10$ to $1000$. 
Then Figure~\ref{fig:dtvmotzheight}(a) shows the plot of $n\ \!\hat{d}_n$, which seems to converge. A coupling argument actually ensures\footnote{We omit details; coupling is first done on the core-size, then on the height. Letting $M(x):=1-2\sum_{k=1}^\infty (4k^2x^2-1)e^{-2k^2x^2}$ be the cumulative function for the height of the Brownian excursion, it  gives $\limsup\ \!n\ \!\hat{d}_n\leq c$, where $c=\frac3{16}\int_{\mathbb{R}}\big|M'(x)+xM''(x)\big| \mathrm{d}x\approx 0.67453$, which agrees with Figure~\ref{fig:dtvmotzheight} and could be the actual limit.} that $\hat{d}_n=O(1/n)$. 
Figure~\ref{fig:dtvmotzheight}(b) shows the plot of $n^2\ \!\hat{d}_n^{\mathrm{rej}}$, which also seems to converge. Again $\hat{d}_n^{\mathrm{rej}}=O(1/n^2)$ should be possible to prove by a coupling argument.  
We note that the better reduction $\hat{d}_n^{\mathrm{rej}}/d_n^{\mathrm{rej}}=\Theta(1/n)$ compared to $\hat{d}_n/d_n=\Theta(1/\sqrt{n})$ should be due to the fact that first order acceleration of convergence better centers core-size at the mode, as seen above.  
In Figure~\ref{fig:dtvmotzheight} the observed fluctuations should be the effect of whether the mode of the (interpolated) height distribution is close to an integer  
(in size $n$ the mode is at a value $h(n)=\Theta(\sqrt{n})$). 
Finally, as a counterpart to Figure~\ref{fig:histogram_core_motz}, Figure~\ref{fig:histogram_height_motz} shows the plot  of $\hat{\pi}_n(h)$ superimposed with the curve of the limit law\footnote{It is the conveniently rescaled law of the height of the Brownian excursion, whose cumulative function is $M(x):=1-2\sum_{k=1}^\infty (4k^2x^2-1)e^{-2k^2x^2}$.},  
and below the plot of $\hat{\pi}_n'(h)-\hat{\pi}_n(h)$ (we can observe a slight drift inherited from the drift for the core-size shown in Figure~\ref{fig:histogram_core_motz}(c)),  
and the plot of $\hat{\pi}_n^{\mathrm{rej}}(h)-\hat{\pi}_n(h)$. 

\begin{figure}
  \centering
  \includegraphics[width=12cm]{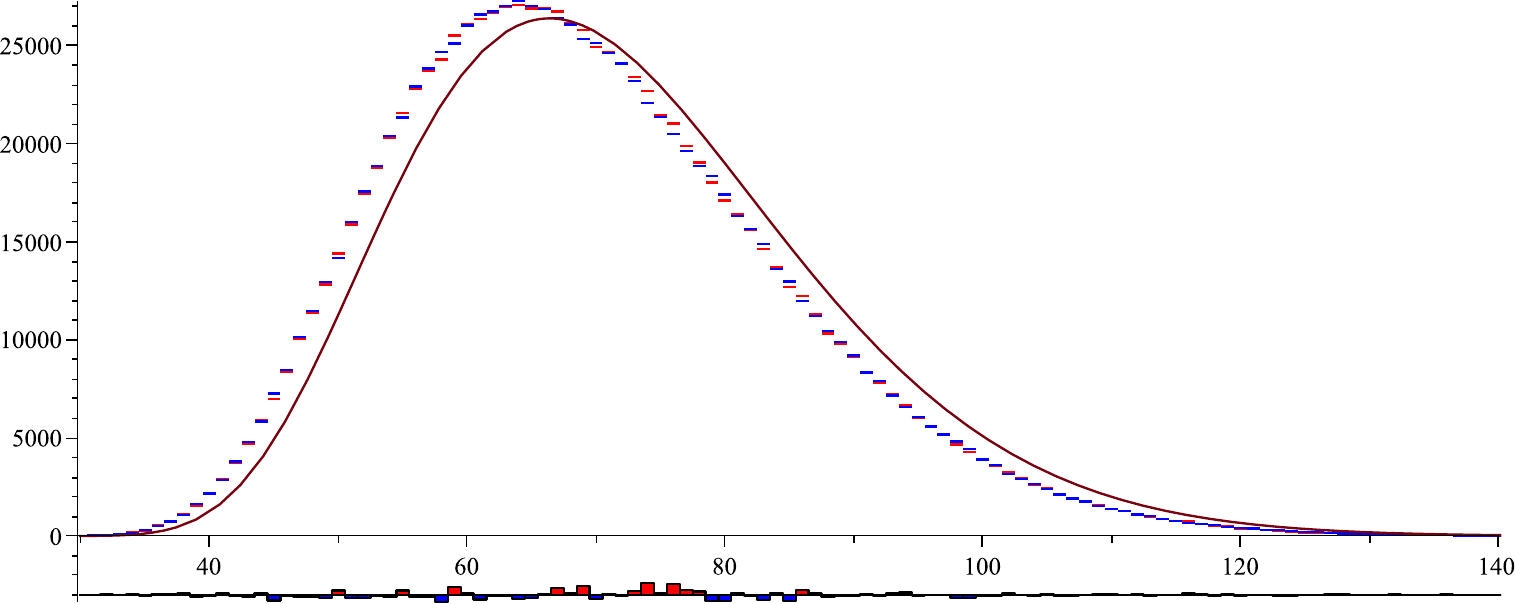}
  \caption{
  Histogram (size $1000$, with $10^6$ samples) for the height of random P\'olya trees under the uniform distribution (red) and the leap distribution (blue), superimposed with the limit curve. The difference between the two histograms is shown below.} 
  \label{fig:histogram_height_polya_1000}
\end{figure}

\begin{figure}
  \centering
  \includegraphics[width=12cm]{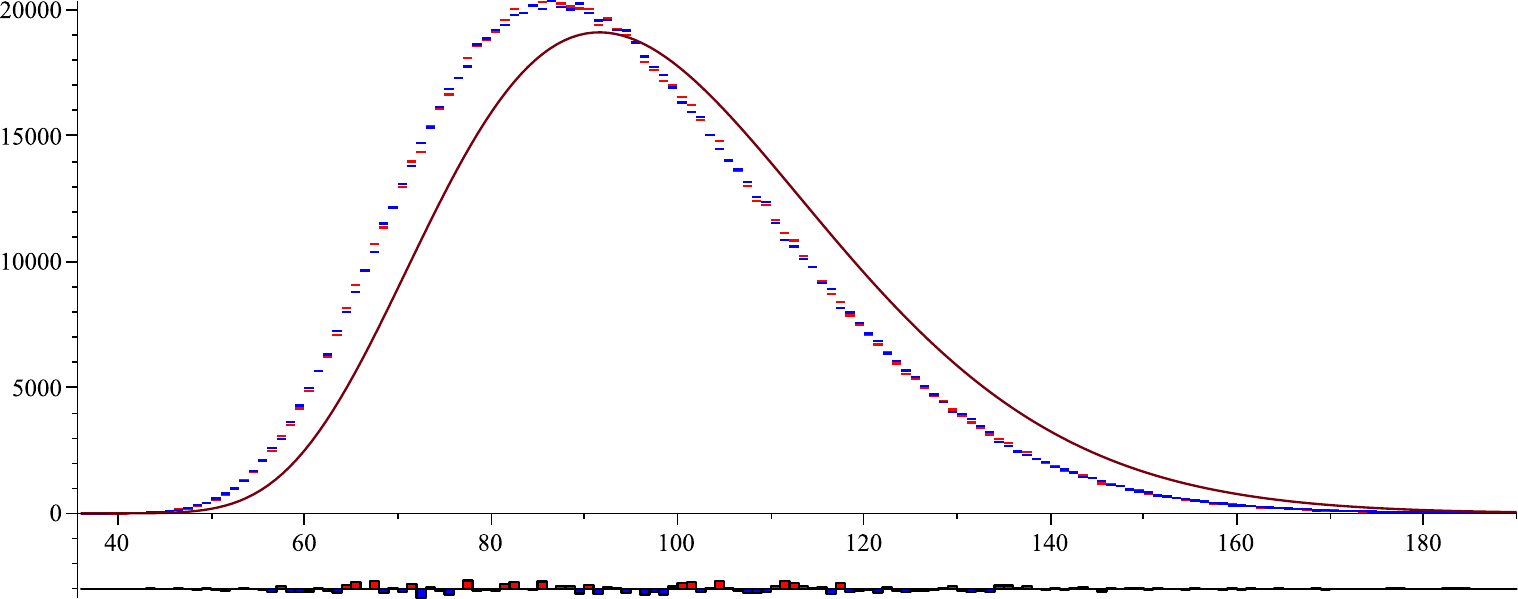}
  \caption{
  Histogram (size $1000$, with $10^6$ samples) for the height of random phylogenetic trees under the uniform distribution (red) and the leap distribution (blue), superimposed with the limit curve. The difference between the two histograms is shown below.} 
  \label{fig:histogram_height_bin_1000}
\end{figure}
\begin{figure}
  \centering
  \includegraphics[width=12cm]{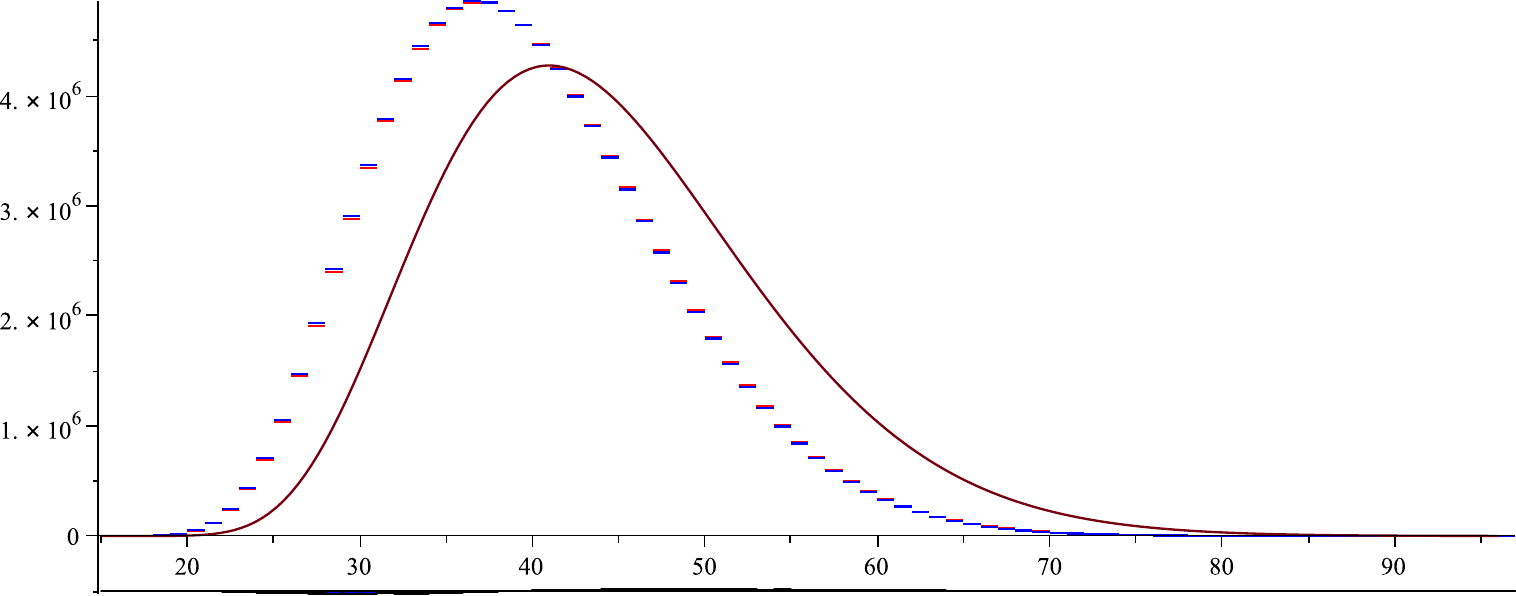}\\[.4cm]
  \includegraphics[width=12cm]{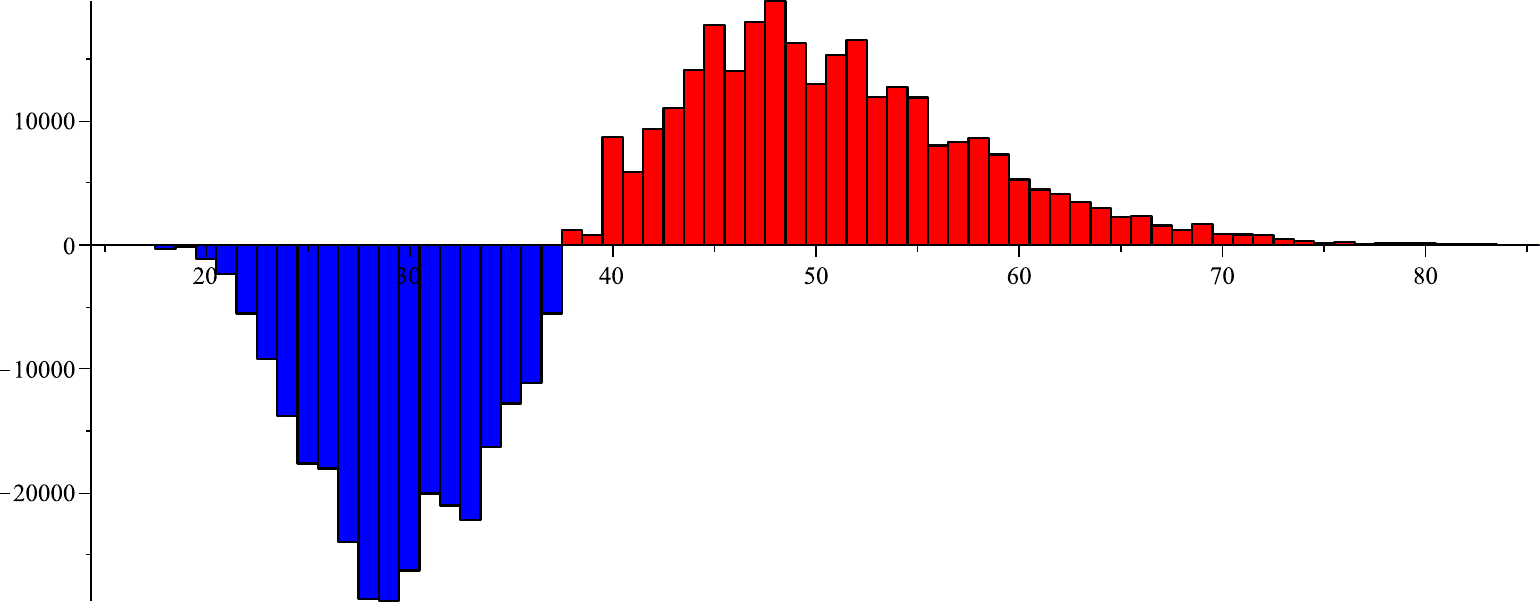}
  \caption{
  Same as Figure~\ref{fig:histogram_height_bin_1000}, in size $200$ with $10^8$ samples. The difference between the two histograms is shown below (unzoomed and zoomed).
  } 
  \label{fig:histogram_height_bin_200}
\end{figure}

For phylogenetic trees and P\'olya trees, exact computations of $\dtv(\pi_n,\pi_n')$ are also possible, but are more time-consuming than for Motzkin walks\footnote{Exact computations of total variation distance for the height parameter are even more out of reach, as the height of a rooted tree endowed with an automorphism is not necessarily equal to the height of its core.}. Instead, we provide histograms based on running the leap generator multiple times at a chosen size, also validating the practical efficiency of our approach. 
Figure~\ref{fig:histogram_height_polya_1000} (resp. Figure~\ref{fig:histogram_height_bin_1000}) show histograms for the height of random P\'olya (resp. phylogenetic) trees at size $1000$ with $10^6$ samples. We see that the two histograms under uniform~\footnote{For P\'olya trees and unlabeled phylogenetic trees, our histograms for the uniform distribution have been obtained by the recursive method of sampling, via the Maple/combstruct package.}  
and leap distribution are very close, whereas the limit density function is more distant (suggesting a quite slow convergence, as already noted in~\cite{bartholdi2024algorithm}), even more for phylogenetic trees. The difference of the two histograms is shown in the lower part of the figure. For P\'olya trees we observe that the height is very slightly drifted to the right under the leap distribution, similarly as for Motzkin walks (Figure~\ref{fig:histogram_height_motz}(b)), whereas for phylogenetic trees positive and negative values alternate in a quite random way. The problem is that the noise due to randomness may dominate the very slight drift effect.  
For phylogenetic trees the drift effect appears more clearly in Figure~\ref{fig:histogram_height_bin_200}, with a smaller size $200$ to accommodate for a much larger number of samples $10^8$, making fluctuations small enough. 

Leap generators also allow us to obtain histograms at larger size than possible with the recursive method. To illustrate convergence to the limit curve, Figure~\ref{fig:histograms_height} shows histograms for several sizes, up to $10^7$, with $10^6$ sampled phylogenetic trees in each case. In large size, the noise due to randomness could be reduced by grouping values into buckets; alternatively one can plot cumulative histograms, as shown in the right column of the figure. 
Figure~\ref{fig:histogram_path_length} then shows histograms for P\'olya trees of another distance-related parameter, namely the average path-length (average of the depths of the nodes of the tree), superimposed with the limit density function, which is explicitly known (see~\cite[Sec.4.2]{bartholdi2024algorithm} and references therein).  

\begin{figure}
  \centering
\includegraphics[width=6cm]{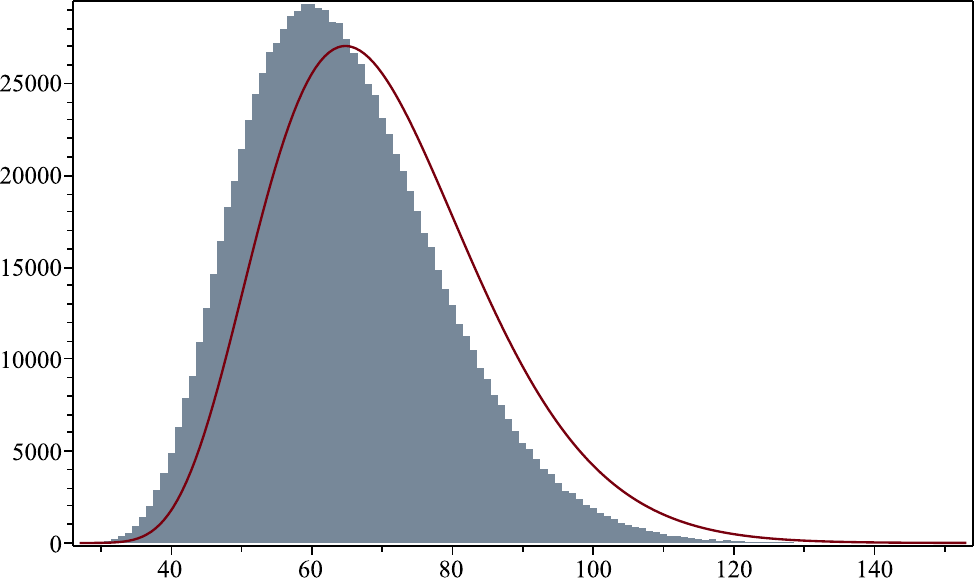}\ \ \includegraphics[width=6cm]{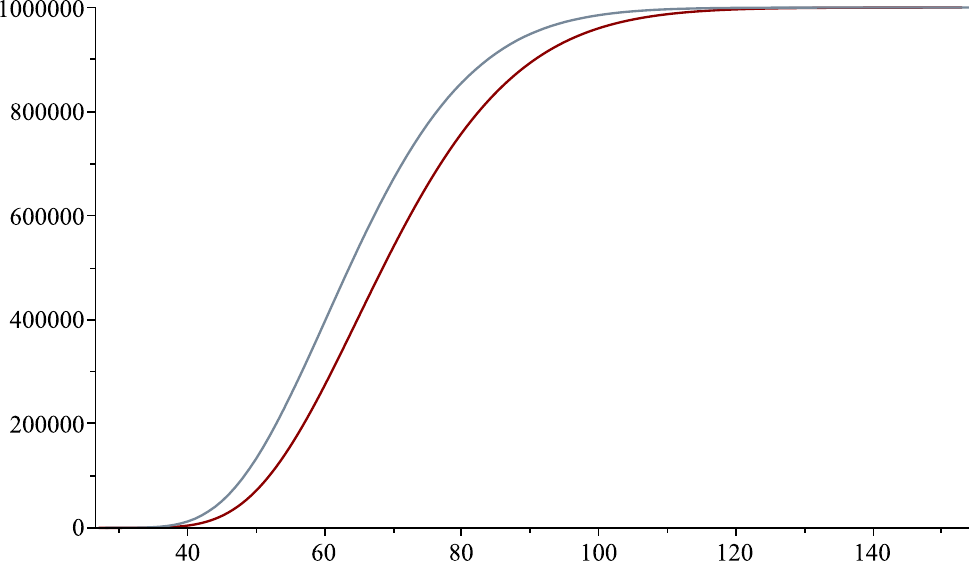}\\[.4cm]
\includegraphics[width=6cm]{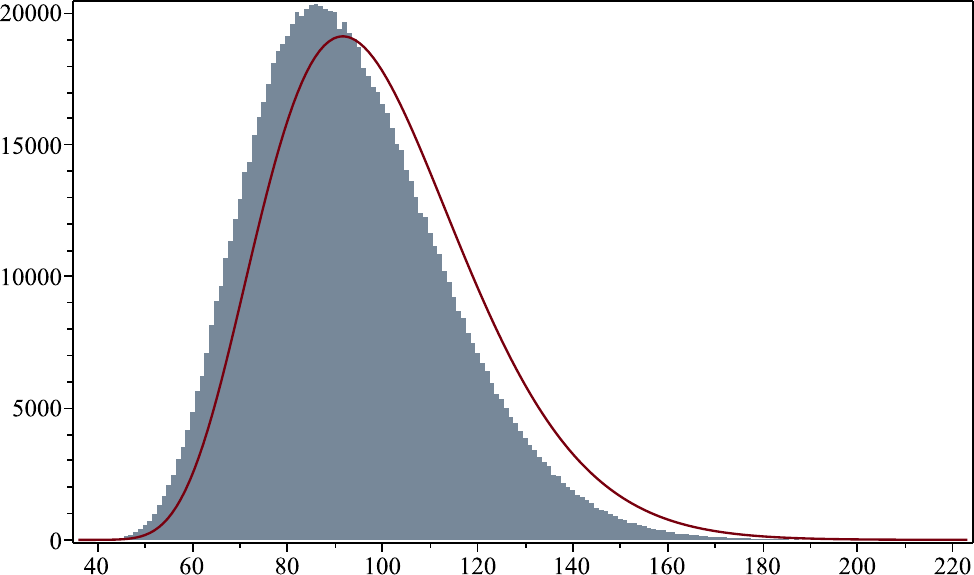}\ \ \includegraphics[width=6cm]{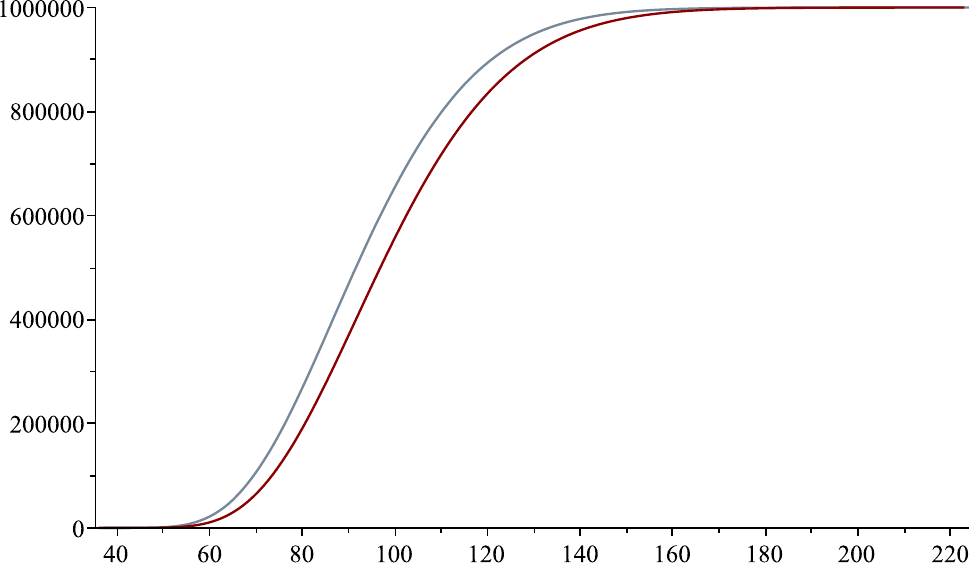}\\[.4cm]
\includegraphics[width=6cm]{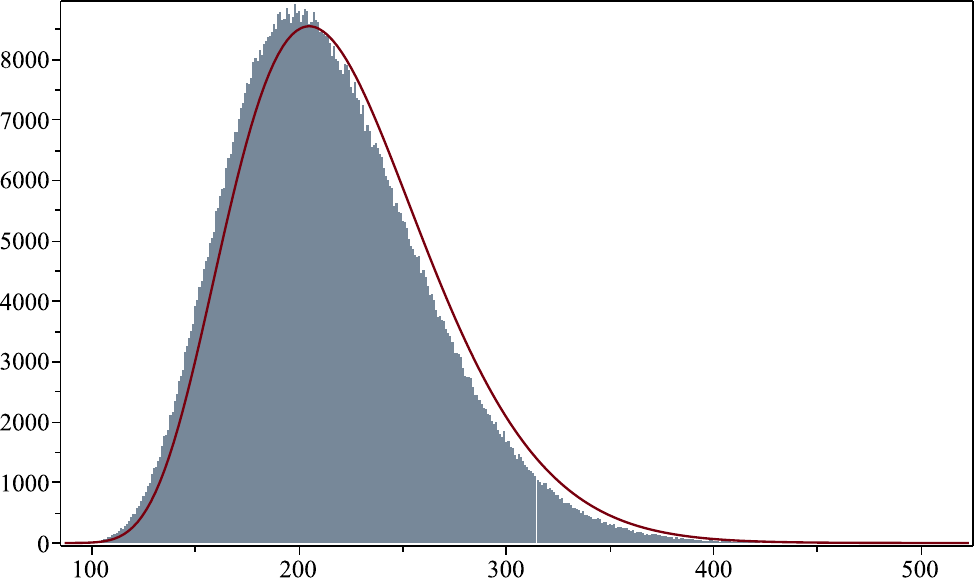}\ \ \includegraphics[width=6cm]{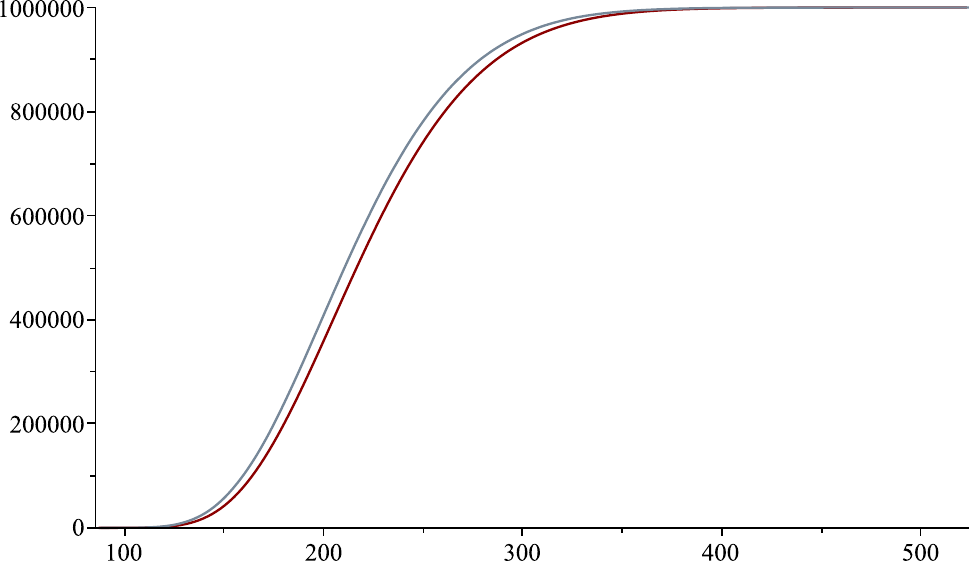}\\[.4cm]
\includegraphics[width=6cm]{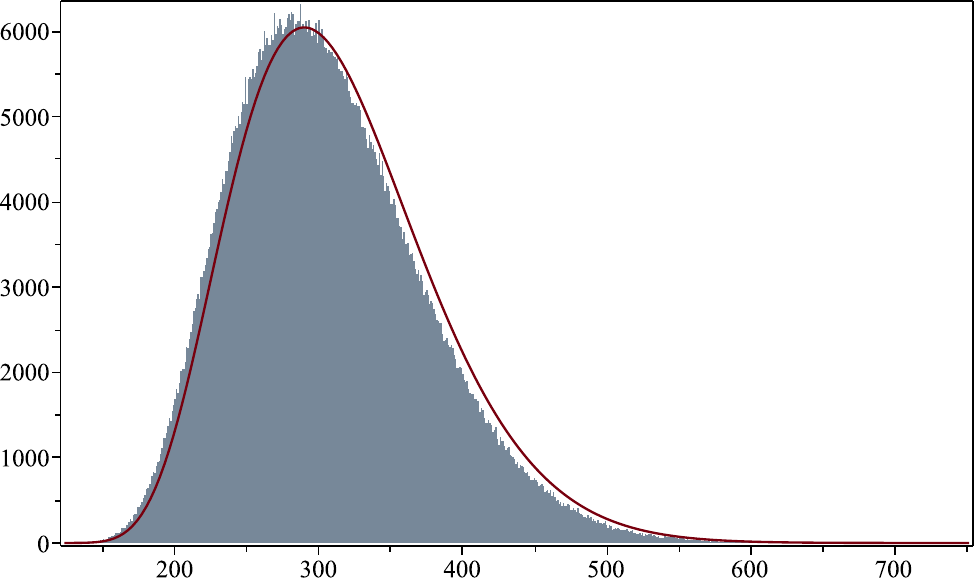}\ \ \includegraphics[width=6cm]{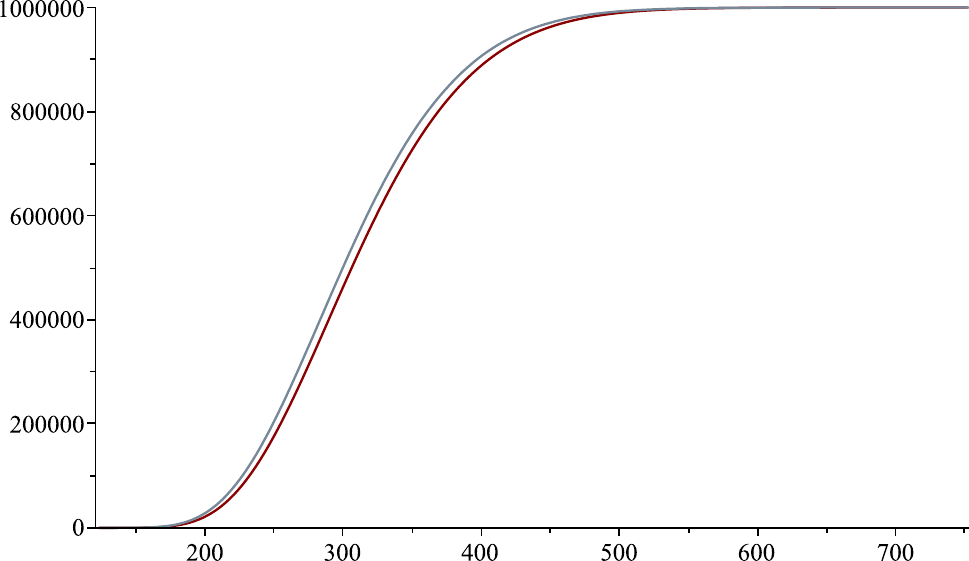}\\[.4cm]
  \caption{Histograms for the height in random phylogenetic trees under the leap distribution, in sizes $500,1000,5000,10000$ from top to bottom, with $10^6$ samples in each case. The non-cumulative (resp. cumulative) histograms are given in the left (resp. right) column. In each case, the histogram is superimposed with the limit density (resp. cumulative) function colored red.} \label{fig:histograms_height}
\end{figure}

Finally, Figure~\ref{fig:leaves_cherries} shows histograms for more local parameters, namely the number of leaves in random P\'olya trees, and the number of cherries (nodes whose two children are leaves) in random phylogenetic trees. 
This time, similarly as in Figure~\ref{fig:histogram_core_motz}(a), each histogram is very close to the limit density function (a normal law whose mean and variance are computed thanks to the quasi-power theorem~\cite[IX.5]{flajolet2009analytic}), suggesting a faster convergence rate.   

For random maps decomposed into blocks (Section~\ref{sec:maps_blocks}), we have tried to compute the total variation distance $d_n=\dtv(\pi_n,\pi_n')$, but at the moment could only carry out\ computations\footnote{Due to the formula~\eqref{eq:pinprime} for the distortion factor, it  requires to count maps with control on all block sizes, not only the size of the root-block.} up to size $50$, which appears as too small to have conclusive plots (further optimization should make it possible to push the computations to higher values).  

\medskip

\begin{figure}
  \centering
\includegraphics[width=12cm]{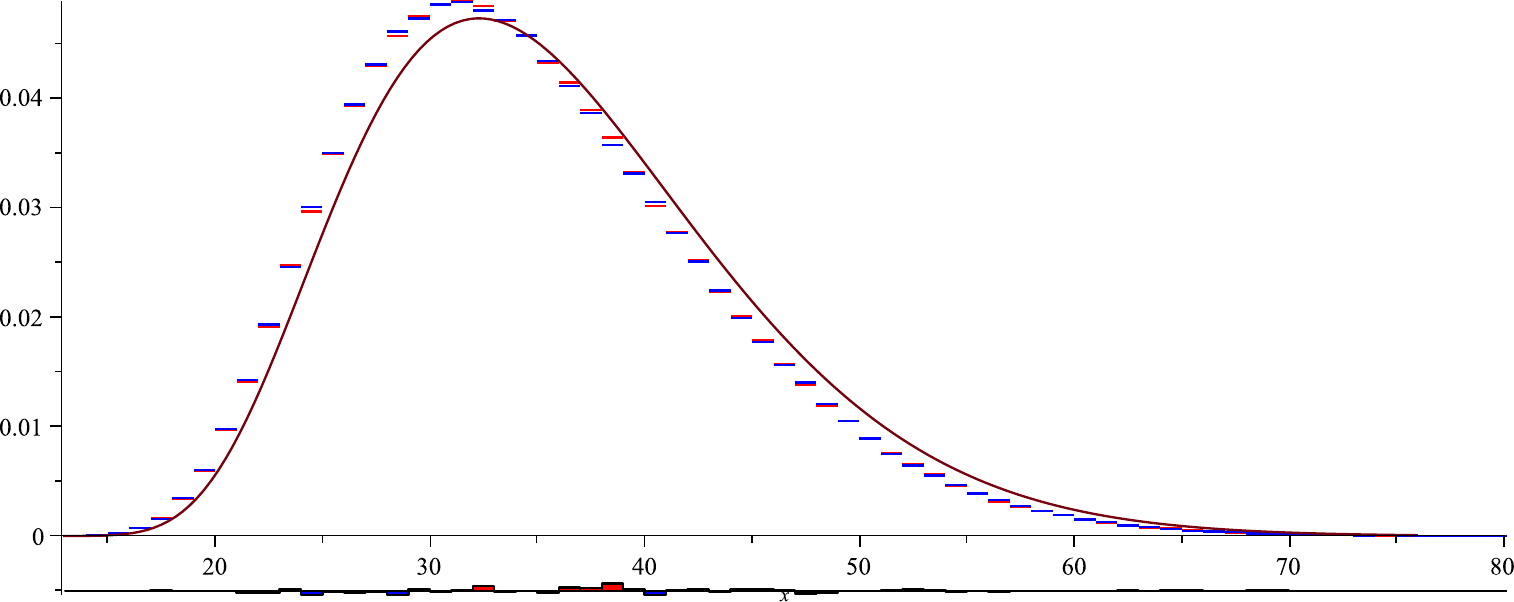}
  \caption{Histograms for the average path-length in random P\'olya trees under the uniform distribution (blue) and the leap distribution (red), in size $10^3$ with $10^6$ samples, superimposed with the limit density function. Samples are grouped into buckets according to the floor of the average path length. The difference between the two histograms is shown below.} 
  \label{fig:histogram_path_length}
\end{figure}

\begin{figure}
  \centering
  \begin{subfigure}[B]{.87\textwidth}
\includegraphics[width=\textwidth]{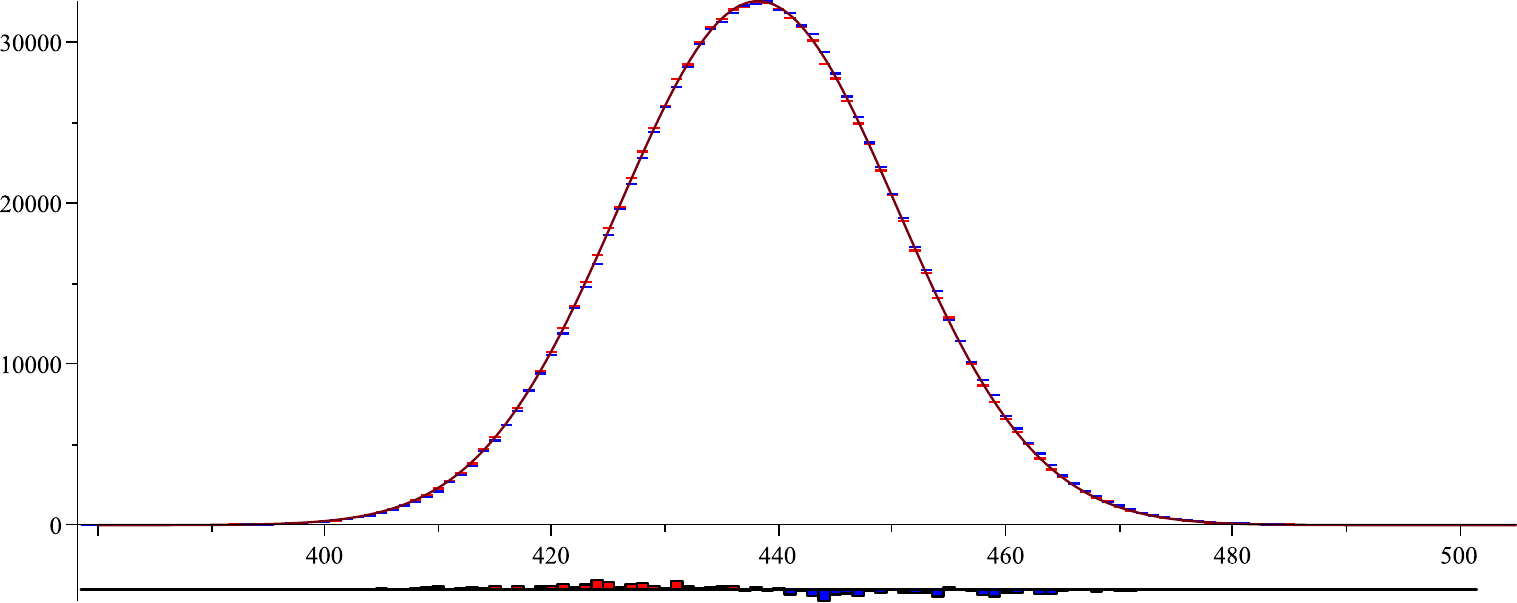}
\end{subfigure}

\medskip

\begin{subfigure}[B]{.87\textwidth}
  \includegraphics[width=\textwidth]{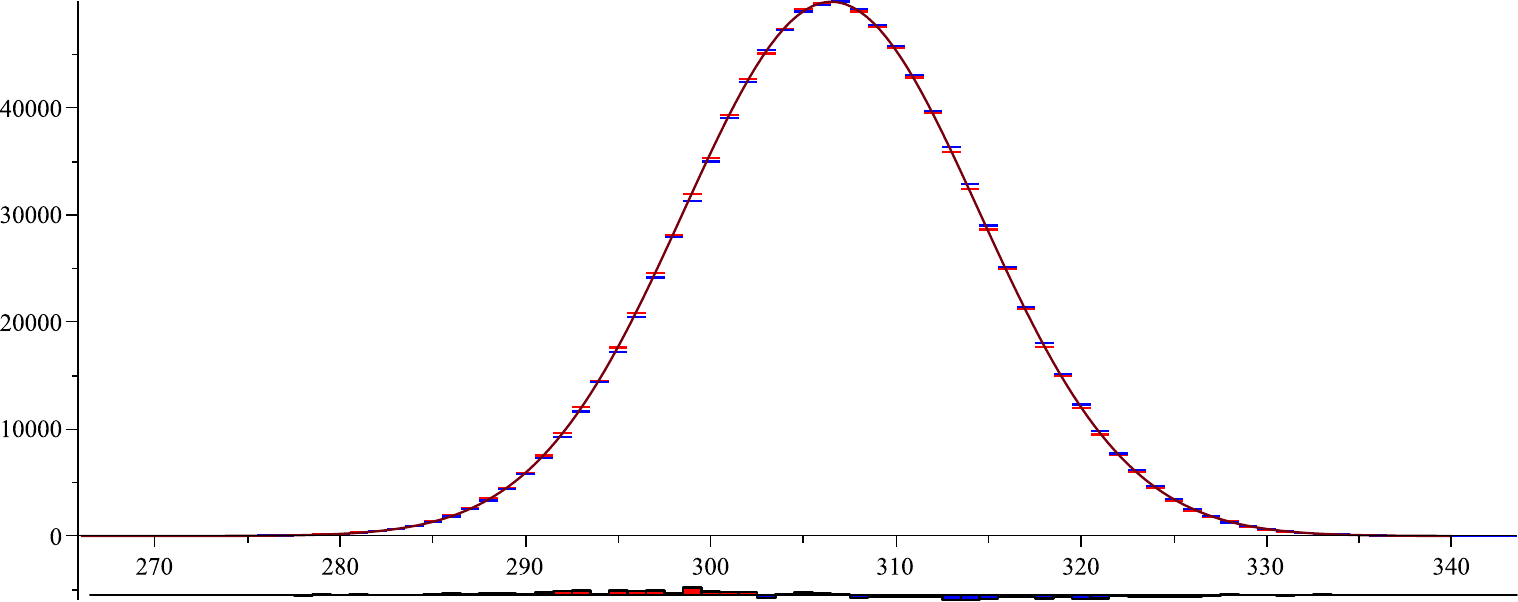}
  \end{subfigure}
   \caption{Histograms for the number of leaves in random P\'olya trees (top) and cherries in random phylogenetic trees (bottom) of size $10^3$ with $10^6$ samples, for the uniform distribution (blue) and the leap distribution (red), superimposed with the Gaussian limit density function.} \label{fig:leaves_cherries}
\end{figure}

\pagebreak

\emph{Acknowledgements.}  The authors are grateful to Persi Diaconis, Konstantinos Panagiotou, and Pablo Rotondo for interesting discussions. First (resp. second) author is partially supported by ANR-23-CE48-0018 CartesEtPlus (resp. ANR-25-CE48-3555 Plasma).

\bibliographystyle{plain}
\bibliography{Biblio}

\end{document}